\documentclass[11pt]{article}
\usepackage{amsfonts,amsmath,amssymb}
\usepackage{amsthm}
\usepackage{enumerate}
\usepackage{graphicx}
\usepackage[hypertex]{hyperref}
\usepackage{epic}
\usepackage{eepic}
\usepackage{tikz}
\usepackage{comment}
\usepackage{multirow}
\usepackage{fouridx,multicol}
\usepackage[all]{xy}

\usepackage{subcaption}
\usepackage{thmtools,thm-restate}

\usepackage{float}%
\usepackage[hypertex]{hyperref}%

\usepackage{comment}
\usepackage{enumerate}

\usepackage{bm}
\makeatletter
    
    \@addtoreset{equation}{section}
  \makeatother

\newcommand{\N }{\mathbb{N}}

\newcommand{\pr}[1]{\mathbb{P}^{#1}}

\theoremstyle{plain} 
\newtheorem{thm}{Theorem}[section]

\newtheorem{fact}[thm]{Fact}

  \newtheorem{prop}[thm]{Proposition} %
  \newtheorem{lemma}[thm]{Lemma} %
\theoremstyle{definition} 
  \newtheorem{df}[thm]{Definition}
\newtheorem*{notation}{Notation}
  \newtheorem{exmp}[thm]{Example}

  \newtheorem{remark}[thm]{Remark}%

\usepackage[truedimen,margin=30truemm]{geometry}
\allowdisplaybreaks

\title{Configuration of five points in $\mathbb P^3$ and their limits}
\author{Naoya SHIMAMOTO \\ Chiba Institute of Technology}
\date{}

\begin{document}

\maketitle

\begin{abstract}
We give a classification of ordered five points in $\mathbb P^3$ under the diagonal action of $GL_4$ over an algebraically closed field of characteristic $0$, by an explicit description of the diagonal action of $GL_4$ on the quintuple of the projective varieties $\mathbb P^3$. 
This is the second simplest setting, where a reductive subgroup of $H$ of $G$ has an open orbit in a (generalised) flag variety $X$ of $G$ but $\#(H\backslash X)=\infty$. 
The closure relations among infinitely many orbits are also given. 
\end{abstract}

\tableofcontents

\section{Introduction}

The object of this article is a description of the orbits of $GL_4$ on the quintuple of the projective spaces $\mathbb P^3$ over an algebraically closed field $\mathbb K$ with characteristic $0$. 
A distinguished feature of our setting is that the group $GL_4$ has an open orbit in $(\mathbb P^3)^5$, however, there are infinitely many orbits of lower dimensions. 
We give an explicit description of these infinitely many orbits by using invariant theory and by combinatorial techniques. 
We also determine the closure relations among orbits. 

Our strategy to describe orbits is divided into the following $2$ steps: to introduce a $GL_4$ equivariant fibration of $(\mathbb P^3)^5$ onto a finite set combinatorially, and further to decompose its fibres via the invariant theory. 

The first step is to give a coarse classification of configurations of $5$ points in $\mathbb P^3$ by using a family of ranks. We define the following map: 
\[\pi\colon (\mathbb P^3)^5\to Map(2^{\{1,2,3,4,5\}},\mathbb N),\ v=([v_i])_{i=1}^5\mapsto\left(I\mapsto \dim\mathrm{Span}\{v_i\}_{i\in I}\right). \]
For an element of $v\in(\mathbb P^3)^5$, we say $\pi(v)\in Map(2^{\{1,2,3,4,5\}},\mathbb N)$ the rank matrix of $v$. 

For $\varphi\in Map(2^{\{1,2,3,4,5\}},\mathbb N)$, we call a maximal element in $\varphi^{-1}(r)\subset 2^{\{1,2,3,4,5\}}$ with respect to the inclusion relations as an $r$-face of $\varphi$ (in particular, we call $1,2$ or $3$-faces as vertices, edges or faces with respectively). We define the rank type of $\varphi\in Map(2^{\{1,2,3,4,5\}},\mathbb N)$ as the tuple of numbers of vertices, edges and faces. We describe the image of $\pi$ in $Map(2^{\{1,2,3,4,5\}},\mathbb N)$ by using rank types of $\varphi\in Map(2^{\{1,2,3,4,5\}},\mathbb N)$. 

\begin{thm}\label{thm:class} A map $\varphi\in Map(2^{\{1,2,3,4,5\}},\mathbb N)$ belongs to the image of the map $\pi\colon (\mathbb P^3)^5\to Map(2^{\{1,2,3,4,5\}},\mathbb N)$ if and only if the rank type of $\varphi$ is one of ones listed in Table \ref{tab:type} (see Theorems \ref{thm:split} and \ref{thm:parameter}). 
\end{thm}
\begin{table}[h]
\centering
\caption{Types of rank matrices in $(\mathbb P^3)^5$}
\label{tab:type}
\begin{tabular}{c|c|c|c|c|c|c}
\begin{tikzpicture}
\fill(0,0) circle (0.06) node[above right]{\scalebox{0.7}{$5$}};
\end{tikzpicture}
&
\begin{tikzpicture}
\fill(0,0) circle (0.06);
\fill(0.5,0.2) circle (0.06) node[above right]{\scalebox{0.7}{$4$}};
\end{tikzpicture}
&
\begin{tikzpicture}
\fill(0,0) circle (0.06) node[above left]{\scalebox{0.7}{$2$}};
\fill(0.5,0.2) circle (0.06) node[above right]{\scalebox{0.7}{$3$}};
\end{tikzpicture}
&
\begin{tikzpicture}
\fill(0,0) circle (0.06) node[above left]{\scalebox{0.7}{$2$}};
\fill(0.25,0.1) circle (0.06);
\fill(0.5,0.2) circle (0.06) node[above right]{\scalebox{0.7}{$2$}};
\draw(-0.25,-0.1)--(0.75,0.3);
\end{tikzpicture}
&
\begin{tikzpicture}
\fill(0,0) circle (0.06) node[above left]{\scalebox{0.7}{$3$}};
\fill(0.25,0.1) circle (0.06);
\fill(0.5,0.2) circle (0.06);
\draw(-0.25,-0.1)--(0.75,0.3);
\end{tikzpicture}
&
\begin{tikzpicture}
\fill(0,0) circle (0.06)  node[above left]{\scalebox{0.7}{$2$}};
\fill(0.2,0.08) circle (0.06);
\fill(0.4,0.16) circle (0.06);
\fill(0.6,0.24) circle (0.06);
\draw(-0.2,-0.08)--(0.8,0.32);
\end{tikzpicture}
&
\begin{tikzpicture}
\fill(0,0) circle (0.06);
\fill(0.15,0.06) circle (0.06);
\fill(0.3,0.12) circle (0.06);
\fill(0.45,0.18) circle (0.06);
\fill(0.6,0.24) circle (0.06);
\draw(-0.2,-0.08)--(0.8,0.32);
\end{tikzpicture}
\\
$(\emptyset)$ & $(2)$a$\times 5$ & $(2)$b$\times 10$ & $(3)$a$\times 15$ & $(3)$b$\times 10$ & $(4)\times 5$  & (5)
 \\\hline
\begin{tikzpicture}
\fill(0,0) circle (0.06) node[above left]{\scalebox{0.7}{$2$}};
\fill(0.6,0) circle (0.06) node[above right]{\scalebox{0.7}{$2$}};
\fill(0.3,0.3) circle (0.06);
\end{tikzpicture}
&
\begin{tikzpicture}
\fill(0,0) circle (0.06) node[above left]{\scalebox{0.7}{$3$}};
\fill(0.6,0) circle (0.06);
\fill(0.3,0.3) circle (0.06);
\end{tikzpicture}
&
\begin{tikzpicture}
\fill(0,0) circle (0.06);
\fill(0.3,0) circle (0.06);
\fill(0.6,0) circle (0.06);
\fill(0.3,0.3) circle (0.06)  node[above]{\scalebox{0.7}{$2$}};
\draw (-0.2,0)--(0.8,0);
\end{tikzpicture}
&
\begin{tikzpicture}
\fill(0,0) circle (0.06)  node[above left]{\scalebox{0.7}{$2$}};
\fill(0.3,0) circle (0.06);
\fill(0.6,0) circle (0.06);
\fill(0.3,0.3) circle (0.06);
\draw (-0.2,0)--(0.8,0);
\end{tikzpicture}
&
\begin{tikzpicture}
\fill(0,0) circle (0.06);
\fill(0.2,0) circle (0.06);
\fill(0.6,0) circle (0.06);
\fill(0.4,0) circle (0.06);
\fill(0.3,0.3) circle (0.06);
\draw (-0.2,0)--(0.8,0);
\end{tikzpicture}
&
\begin{tikzpicture}
\fill(0,0) circle (0.06)node[above left]{\scalebox{0.7}{$2$}};
\fill(0.6,0) circle (0.06);
\fill(0.3,0.51) circle (0.06);
\fill(0.3,0.17) circle (0.06);
\end{tikzpicture}
&
\begin{tikzpicture}
\fill(0,0) circle (0.06);
\fill(0.3,0) circle (0.06);
\fill(0.6,0) circle (0.06);
\fill(0.2,0.2) circle (0.06);
\fill(0.4,0.4) circle (0.06);
\draw (-0.2,0)--(0.8,0);
\draw (-0.15,-0.15)--(0.55,0.55);
\end{tikzpicture}
\\
$(3,3)$a$\times 15$ & $(3,3)$b$\times 10$ & $(4,4)$a$\times 10$ & $(4,4)$b$\times 30$ & $(5,5)\times 5$ & $(4,6)\times 10$ & $(5,6)\times 15$
 \\\hline
\begin{tikzpicture}
\fill(0.25,0) circle (0.06);
\fill(0.55,0) circle (0.06);
\fill(0.85,0) circle (0.06);
\fill(0.2,0.2) circle (0.06);
\fill(0.4,0.4) circle (0.06);
\draw (-0.2,0)--(1,0);
\draw (-0.15,-0.15)--(0.55,0.55);
\end{tikzpicture}
&
\begin{tikzpicture}
\fill(0,0) circle (0.06);
\fill(0.6,0) circle (0.06);
\fill(0.3,0.51) circle (0.06);
\fill(0.3,0.18) circle (0.06);
\draw(-0.1,-0.17)--(0.4,0.68);
\draw(0.7,-0.17)--(0.2,0.68);
\draw(-0.2,-0.12)--(0.6,0.36);
\draw(0.8,-0.12)--(0,0.36);
\draw(-0.2,0)--(0.8,0);
\draw(0.3,-0.2)--(0.3,0.7);
\fill(0.8,0.2) circle (0.06);
\end{tikzpicture}
&
\scalebox{0.5}{
\begin{tikzpicture}
\fill(-1,0,0) circle (0.12) node[above left]{\scalebox{1.4}{$2$}};
\fill(0,{sqrt(2)},0) circle (0.12);
\fill(1/2,0,{sqrt(3)/2}) circle (0.12);
\fill(1/2,0,{-sqrt(3)/2}) circle (0.12);
\draw (-1,0,0)--(0,{sqrt(2)},0)--(1/2,0,{-sqrt(3)/2})--(1/2,0,{sqrt(3)/2})--(-1,0,0);
\draw (0,{sqrt(2)},0)--(1/2,0,{sqrt(3)/2});
\draw[dashed] (-1,0,0)--(1/2,0,{-sqrt(3)/2});
\end{tikzpicture}}
&
\scalebox{0.5}{
\begin{tikzpicture}
\fill(-1,0,0) circle (0.12);
\fill(0,{sqrt(2)},0) circle (0.12);
\fill(1/2,0,{sqrt(3)/2}) circle (0.12);
\fill(1/2,0,{-sqrt(3)/2}) circle (0.12);
\fill(-1/4,0,{sqrt(3)/4}) circle (0.12);
\draw (-1,0,0)--(0,{sqrt(2)},0)--(1/2,0,{-sqrt(3)/2})--(1/2,0,{sqrt(3)/2})--(-1,0,0);
\draw (0,{sqrt(2)},0)--(1/2,0,{sqrt(3)/2});
\draw[dashed] (-1,0,0)--(1/2,0,{-sqrt(3)/2});
\end{tikzpicture}}
&
\scalebox{0.5}{
\begin{tikzpicture}
\fill(-1,0,0) circle (0.12);
\fill(0,{sqrt(2)},0) circle (0.12);
\fill(1/2,0,{sqrt(3)/2}) circle (0.12);
\fill(1/2,0,{-sqrt(3)/2}) circle (0.12);
\draw (-1,0,0)--(0,{sqrt(2)},0)--(1/2,0,{-sqrt(3)/2})--(1/2,0,{sqrt(3)/2})--(-1,0,0);
\draw (0,{sqrt(2)},0)--(1/2,0,{sqrt(3)/2});
\draw[dashed] (-1,0,0)--(1/2,0,{-sqrt(3)/2});
\fill[lightgray] (-1,0,0)--(1/2,0,{sqrt(3)/2})--(0,{sqrt(2)},0)--cycle;
\fill(-1/6,{sqrt(2)/3},{sqrt(3)/6}) circle (0.12);
\end{tikzpicture}}
&
\scalebox{0.5}{
\begin{tikzpicture}
\fill(-1,0,0) circle (0.12);
\fill(0,{sqrt(2)},0) circle (0.12);
\fill(1/2,0,{sqrt(3)/2}) circle (0.12);
\fill(1/2,0,{-sqrt(3)/2}) circle (0.12);
\fill(1,1.5,1) circle (0.12);
\draw (-1,0,0)--(0,{sqrt(2)},0)--(1/2,0,{-sqrt(3)/2})--(1/2,0,{sqrt(3)/2})--(-1,0,0);
\draw (0,{sqrt(2)},0)--(1/2,0,{sqrt(3)/2});
\draw[dashed] (-1,0,0)--(1/2,0,{-sqrt(3)/2});
\end{tikzpicture}}
\\
$(5.8)\times 10$ & $(5,10)$ & $(4,6,4)\times 10$ & $(5,8,5)\times 10$ & $(5,10,7)\times 5$ & $(5,10,10)$ 
\end{tabular}
\end{table}

\begin{remark} 
\begin{enumerate}
\item For $v=[v_i]_{i=1}^5\in (\mathbb P^3)^5$, the type of $\pi(v)$ is $(A_1,A_2,A_3)$ if  
\[A_r=\#\left\{\langle v_i\rangle_{i\in I}\subset \mathbb K^n\left|\ I\subset \{1,2,3,4,5\},\ \dim\langle v_i\rangle_{i\in I}=r\right.\right\}\ \ \ (r=1,2,3). \]
For instance, $(A_1,A_2,A_3)=(5,10,7)$ in Table \ref{tab:type} implies that there are $5$ black points, $10$ edges connecting the black points, and $7$ faces respectively. 
\item The symbols in Table \ref{tab:type} visualise the configurations in some coordinate $\mathbb K^3\hookrightarrow \mathbb P^3$. 
For instance, if the rank matrix of $v=[v_i]_{i=1}^5$ is of type $(4,4)$a, then $\dim\langle v_i\rangle_{i=1}^5=3$, a unique distinct triple shares a $2$-dimensional space, and the rest two points coincide. 
\item The list of Table \ref{tab:type} coincides with the list of configurations of unordered $5$ points in $\mathbb P^3$. The number written after each type denotes the number of configurations of \emph{ordered} $5$ points. 
To distinguish rank matrices which have the same type but are different as configurations, we mark on them like types $(3,3)$a and $(3,3)$b. 
\end{enumerate}
\end{remark}
The second step is to describe $G$-orbits in each fibres of $\pi$. Since the map $\pi\colon (\mathbb P^3)^5\to Map(2^{\{1,2,3,4,5\}},\mathbb N)$ is $GL_4$-invariant, if we describe the orbits in each fibre $\pi^{-1}(\varphi)$ for $\varphi\in\mathrm{Image}(\pi)$ classified in Theorem \ref{thm:class}, we obtain a complete description of $\mathrm{diag}(GL_4)\setminus(\mathbb P^3)^5$. 

\begin{thm}\label{thm:orbit} For each map $\varphi\in Map(2^{\{1,2,3,4,5\}},\mathbb N)$ classified in Theorem \ref{thm:class}, consider the fibre $\pi^{-1}(\varphi)$ in $(\mathbb P^3)^5$ consisting of all elements with the rank matrices $\varphi$. 
\begin{enumerate}
\item If $\varphi$ is \emph{not} of types $(4)$, $(5)$, $(5,5)$, $(5,8)$, or $(5,10)$, then $\pi^{-1}(\varphi)$ is a single $G$-orbit (see Theorems \ref{thm:split} and \ref{thm:parameter})
\item If $\varphi$ is one of the types $(4)$, $(5)$, $(5,5)$, $(5,8)$, or $(5,10)$, then $ G\backslash \pi^{-1}(\varphi)$ is bijective with some open dense subset of $\mathbb P^1$, $\mathbb P^1\times\mathbb P^1$, $\mathbb P^1$, $\mathbb P^1$, or $\mathbb P^2$ respectively (see Theorem \ref{thm:parameter}). 
\end{enumerate}
\end{thm}

According to this description of all orbits, we can characterise the closure relations by using the following  partial orders defined in $\mathrm{Image}(\pi)\subset Map(2^{\{1,2,3,4,5\}},\mathbb N)$ by using combinatorial properties of them. 
\begin{df} \label{df:partial}
For the subset $\mathrm{Image}(\pi)$ of $Map(2^{\{1,2,\ldots,m\}},\mathbb N)$, we define partial orders $\leq$, $\preceq$, and $\prec$ as below. 
\begin{enumerate}
\item We write $\psi\leq\varphi$ for $\varphi,\psi\in\mathrm{Image}(\pi)$ if $\psi(I)\leq \varphi(I)$ for all $I\subset\{1,2,\ldots,m\}$. 
\item Define $\prec$ as the partial order generated by the relation $\psi\prec\varphi$ for $\psi,\varphi\in\mathrm{Image}(\pi)$ which satisfy the following conditions: $\rho(\varphi)=\{I_k^{r_k}\}_{k=1}^l$, $\rho(\psi)=\{I_k^{s_k}\}_{k=1}^l$ (see Definition \ref{df:rho}), $\psi\leq \varphi$, and there exists an integer $1\leq k\leq l$ such that $s_k<r_k$. 
\item Define $\preceq$ as the partial order generated by the relation $\psi\preceq\varphi$ for $\psi,\varphi\in\mathrm{Image}(\pi)$ which satisfy one of the following conditions: 
\begin{enumerate}
\item There exists an $r$-face $J\subset\{1,2,\ldots,m\}$ such that $\psi(I)=\varphi(I\cup J)+\varphi(I\cap J)-\varphi(J)$ hlods for all $I\subset\{1,2,\ldots,m\}$. 
\item $\rho(\varphi)=\{I_k^{r_k}\}_{k=1}^l$ (see Definition \ref{df:rho}), $\psi\leq\varphi$, and $\psi|_{I_k}= \varphi|_{I_k}$ for all $1\leq k\leq l$. 
\item $\psi\prec\varphi$. 
\end{enumerate}
\end{enumerate}
\end{df}
\begin{thm} \label{thm:closure} For $\varphi\in \mathrm{Image}(\pi)\subset Map(2^{\{1,2,\ldots,m\}},\mathbb N)$ and an orbit $\mathcal O\subset \pi^{-1}(\varphi)\subset (\mathbb P^3)^5$, we have the following criterions for $\psi\in\mathrm{Image}(\pi)$ whether $\pi^{-1}(\psi)\subset(\mathbb P^3)^5$ is in the closure of the orbit $\mathcal O$ (see Theorems \ref{thm:clsingle} and \ref{thm:clparameter}): 
\begin{enumerate}
\item If $\mathcal O=\pi^{-1}(\varphi)$ (see the classification in Theorem \ref{thm:orbit}), then $\overline{\mathcal O}=\coprod_{\psi\leq \varphi}\pi^{-1}(\psi)$. 
\item If $\mathcal O\subsetneq \pi^{-1}(\varphi)$, then 
\begin{enumerate}
\item $\pi^{-1}(\psi)\cap \overline{\mathcal O}\neq\emptyset$ if and only if $\psi\preceq\varphi$; 
\item $\pi^{-1}(\psi)\subset \overline{\mathcal O}$ if and only if $\psi\preceq\varphi$ and $\pi^{-1}(\psi)$ is a single orbit, or $\psi\prec\varphi$. 
\end{enumerate}
\end{enumerate}
\end{thm}

Our results may be considered as a special case of the description of the double coset $H\backslash G/P$ where $P$ is a parabolic subgroup of a reductive group $G$, and $H$ is a reductive subgroup of $G$. 
An explicit description of the double coset has been extensively studied when $\#(H\backslash G/P)<\infty$, for instance, in the setting where $(G,H)$ is a symmetric pair by Matsuki \cite{mp,mpp} which generalises the Bruhat decomposition, and in the non-symmetric setting where $(G,H)=(G'\times \cdots\times  G',\mathrm{diag}(G'))$ in \cite{mwzA,mwzC,mBex,mB}, and where $(G,H)=(O(p+1,q)\times O(p,q),\mathrm{diag}O(p,q))$ by Kobayashi-Leontiev \cite{kl}.

Such a decomposition is in deep connection with representation theory.
For instance, the existence of open orbits, and more strongly, the finiteness of orbits play a crucial role in the estimate of multiplicities in both induction and restriction as was shown in Kobayashi-Oshima \cite{ko}, and also in \cite{k,tt}. See also the classification of Kobayashi-Matsuki \cite{km}. 
Moreover, the orbit decomposition in the case \cite{kl} is responsible for the ``family'' of symmetry breaking operators from $O(p+1,q)$-modules to $O(p,q)$-modules \cite{kl,ks,ks2}. So far, a description of the double coset was known only where $\#(H\backslash G/P)<\infty$, referred to as ``finite type''. This condition is equivalent to the existence of open $H$-orbit on $G/P$ if $P$ is a minimal parabolic subgroup of $G$ by Brion, Vinberg, and Matsuki; \cite{br1,br2,morb}, but it is not always the case if $P$ is a general parabolic subgroup. See also some classifications of ``finite type'' for $(G,H)=(G'\times \cdots\times  G',\mathrm{diag}(G'))$ in \cite{mwzA,mwzC,mB}.

\bigskip
Note that the pair $(G,H)=(G'^m, \mathrm{diag}(G'))$ is no more a symmetric pair if $m\geq 3$, and is of ``finite type'' only if $m\leq 3$. 
On the other hand, the existence of open $H$-orbits on $G/P$ gives useful information on the branching laws between representations of $H$ and representations of $G$ induced from characters of $P$ \cite{k,tt}. 


\bigskip
Our setting in this article corresponds to the case where $G'=GL_n$ and $G'/P'\simeq \mathbb P^{n-1}$. We have shown in \cite{s} that the existence of both an open orbit and infinitely many orbits is equivalent to $4\leq m\leq n+1$. An explicit description of ``generic'' orbits is given in \cite{s} for general $(n,m)$, whereas a complete description of $\mathrm{diag}(G')$-action and closure relations among orbits is given for the simplest case where $(n,m)=(3,4)$ in \cite{s1}. In this article, we develop these techniques for the second simplest case where $(n,m)=(4,5)$.

\begin{notation}
We set $\N=\{0,1,2,3,\ldots\}$, and $[m]$ denotes the set $\{1,2,\ldots,m \}\subset \mathbb N$ for $m\in \mathbb N$. 

We let $\mathbb K$ be an algebraically closed field with characteristic $0$. For a vector $v\in \mathbb K^n\setminus\{\bm 0\}$, we write the $\mathbb K^\times$-orbit through $v$ as $[v]\in\pr{n-1}$. Similarly for a matrix $(v_i)_{i=1}^m\in M(n,m)$ without any columns equal to $\bm 0$, the notion $[v_i]_{i=1}^m$ denotes the $m$-tuple of $\mathbb P^{n-1}$. For an $m$-tuple $[v_i]_{i=1}^m$ in $\mathbb P^{n-1}$, we write the subspace spanned by these $m$ vectors by $\langle v_i\rangle_{i=1}^m$. Furthermore, $\{e_i\}_{i=1}^n$ denotes the standard basis of $\mathbb K^n$. 
\end{notation}

\section{Classification of orbits}
\label{orbitdecomposition}


\subsection{Indecomposable splittings and rank matrices}

To simplify the description of ${GL_n}$-orbit decomposition of $(\mathbb P^{n-1})^m$, we introduce two ${GL_n}$-invariant maps from $(\mathbb P^{n-1})^m$ with finite images parametrised combinatorially. Then we obtain the decomposition of $(\mathbb P^{n-1})^m$ into a finite number of $GL_n$-invariant fibres via these maps. 

\bigskip
The first notion is the indecomposable splitting. Identify an element $v=[v_i]_{i=1}^m\in (\mathbb P^{n-1})^m$ as a representation of the star shaped quiver $T$ with $m$ edges of the length $1$ equipped the linear inclusions $\mathbb Kv_i\hookrightarrow\mathbb K^n$. Then we can consider its splitting into indecomposable ones. Remark that it is unique up to isomorphisms, because the category consisting of all representations with linear inclusions on any edges of $T$ is a full-subcategory of an Abelian category consisting of all representations of $T$. 

Let $v$ is decomposed as $v=w\oplus u$ where $w$ is a representation of $T$ with a linear space $W$ on its centre vertex, and linear inclusions into $W$ from subspaces $\{W_i\}_{i=1}^m$ on the edges (resp. $u$). Then we have $\mathbb K^n=W\oplus U$ and $\mathbb Kv_i=W_i\oplus U_i$. It leads that, there exists a partition $[m]=I_W\amalg I_U$ such that $\langle v_i\rangle_{i\in I_W}\subset W$ and $\langle v_i\rangle_{i\in I_U}\subset U$. Hence we can define as below:

\begin{df}\label{df:splitting} \begin{enumerate}
\item For $v=[v_i]_{i=1}^m\in (\mathbb P^{n-1})^m$, we say $v$ is essentially indecomposable if there do not exist any partition $I\amalg J=[m]$ such that $I,J\neq\emptyset$ and $\langle v_i\rangle_{i\in I}\cap \langle v_j\rangle_{j\in J}=\bm 0$. 
\item Define a set $\mathcal P_{m}$ as bellow:
\[\mathcal P_{m}:=\left\{\{I_k^{r_k}\}_{k=1}^l\left|\ \coprod_{k=1}^lI_k=[m],\ I_k\neq\emptyset,\ r_k\in\mathbb N\right.\right\}.\]
\item Define the map $\varpi\colon (\mathbb P^{n-1})^m\to \mathcal P_{m},\ v\mapsto \{I_k^{r_k}\}_{k=1}^l$ where 
\begin{enumerate}
\item $\langle v_i\rangle_{i=1}^m=\bigoplus_{k=1}^l \langle v_i\rangle_{i\in I_k}$ and  $r_k=\dim\langle v_i\rangle_{i\in I_k}$, 
\item $[v_i]_{i\in I_k}\in (\mathbb P^{n-1})^{\#I_k}$ is essentially indecomposable. 
\end{enumerate}
\item For $\varpi(v)=\{I_k^{r_k}\}_{k=1}^l$, we call the tuple $\{(\#I_k)^{r_k}\}_{k=1}^l$ the splitting type of $v$. 
\end{enumerate}
\end{df}
\begin{exmp} The notion $\varpi(v)=\{I_k^{r_k}\}_{k=1}^l$ is equivalent to that $v$ splits into indecomposable representations of $T$ with the $r_k$-dimensional space on the centre vertex, and $1$-dimensional spaces on the $i$-th vertices for $i\in I_k$, and $0$-spaces on others. For instance, if $(n,m)=(3,4)$, we have 
\begin{align*}
\varpi([e_1,e_1,e_2,e_3])&=\left\{\{1,2\}^1, \{3\}^1,\{4\}^1\right\},&& \textrm{splitting type }\{2^1,1^1,1^1\}, \\
\varpi([e_1,e_2,e_3,e_2+e_3])&=\left\{\{1\}^1,\{2,3,4\}^2\right\},&&\textrm{splitting type }\{1^1,3^2\}. 
\end{align*}
\end{exmp}
Remark that the map $\varpi$ is $GL_n$-invariant since $GL_n$-actions on $(\mathbb P^{n-1})^m$ coincide with the isomorphisms between representations of the quiver $T$. 
With this notion of indecomposable splittings, we can generically describe $GL_n$-orbits on $(\mathbb P^{n-1})^m$ as follows: 
\begin{fact}[\cite{kac,s}] \label{thm:splittingdecomp} For the map $\varpi\colon X_{n,m}=(\mathbb P^{n-1})^m\to \mathcal P_m$ defined in Definition \ref{df:splitting}, 
\begin{enumerate}
\item the image of the map $\varpi$ is the set 
\[\left\{\{I_k^{r_k}\}_{k=1}^l \in \mathcal P_m\left|\ r_k=1\textrm{ or } 2\leq r_k\leq \#I_k-1 \textrm{ for each }1\leq k\leq l,\textrm{ and }\sum_{k=1}^lr_k\leq n\right.\right\}, \]
\item for an element $\{I_k^{r_k}\}_{k=1}^l\in\mathrm{Image}(\varpi)$, there exists the following open dense embedding: 
\[\prod_{k=1}^l (\mathbb P^{r_k-1})^{\#I_k-1-r_k}  \hookrightarrow  GL_n\setminus \varpi^{-1}\left(\{I_k^{r_k}\}_{k=1}^l\right),\ \ \ 
p=\left(p_k\right)_{k=1}^l\mapsto GL\cdot v_p\]
where $v_p=[v_i]_{i=1}^m$ is defined as follows under the notation $R(k)=\sum_{k'<k}r_{k'}$: 
\[[v_{i}]_{i\in I_k}=\left[e_{R(k)+1},e_{R(k)+2},\ldots,e_{R(k)+r_k},\sum_{j=1}^{r_k}e_{R(k)+j},p_k\right],\] 
\item for an element $\{I_k^{r_k}\}_{k=1}^l$ of the subset 
\[\mathcal P_{n,m}':=\left\{\{I_k^{r_k}\}_{k=1}^l \in \mathrm{Image}(\varpi)\left|\ r_k=1\textrm{ or } 2\leq r_k= \#I_k-1 \textrm{ for each }1\leq k\leq l\right.\right\} \]
of $\mathrm{Image}(\varpi)$, the open dense embeddings in (2) is bijective. 
In particular, $\varpi^{-1}\left(\{I_k^{r_k}\}_{k=1}^l\right)\subset (\mathbb P^{n-1})^m$ is a single $GL_n$-orbit,
\end{enumerate}
\end{fact}

According to Fact \ref{thm:splittingdecomp}, we can describe all orbits in the fibres of $\varpi$ on $\mathcal P_{n,m}'\subset \mathrm{Image}(\varpi)$. However, we can determine other orbits only generically. To solve this problem, we introduce another notion of $GL_n$-invariance. 
The second notion is rank matrices. We define as below: 
\begin{df}\label{df:rankmatrix} For $(\mathbb P^{n-1})^m$, we define a $GL_n$-invariant map as follows: 
\begin{equation*} 
\pi\colon  (\mathbb P^{n-1})^m  \to Map(2^{[m]},\mathbb N), \ [v_i]_{i=1}^m \mapsto  \left(I\mapsto \dim\langle v_i\rangle_{i\in I}\right). 
\end{equation*}
Furthermore, for $\varphi\in Map(2^{[m]},\mathbb N)$ and an integer $r$, we call a maximal element of $\varphi^{-1}(r)\subset 2^{[m]}$ under the inclusion relation as an $r$-face of $\varphi$, and the tuple $(\#\{\textrm{$r$-faces of $\varphi$}\})_{r=1}^\infty$ as the rank type of $\varphi$. We also define $r$-faces and the rank type of $v\in(\mathbb P^{n-1})^m$ or the orbit through it as those of $\pi(v)\in Map(2^{[m]},\mathbb N)$. 
\end{df}
\begin{exmp} For $v=[e_1,e_2,e_3,e_2+e_3]$ and $\pi(v)=\varphi$, we have
\begin{align*}
\varphi^{-1}(1)&=\left\{\{1\},\{2\},\{3\},\{4\}\right\} \\ 
\varphi^{-1}(2)&=\left\{\{1,2\},\{1,3\},\{1,4\}.\{2,3\},\{2,4\},\{3,4\},\{2,3,4\}\right\} \\
\varphi^{-1}(3)&=\{\{1,2,3\},\{1,2,4\},\{1,3,4\},\{1,2,3,4\}\}, 
\end{align*}
and the rank type of $v$ is $(4,4,1,0,0,\ldots)$. We may simply write $(4,4)$ when the last $(1,0,0,\ldots)$ is clear. 
\end{exmp}
We observe the relationship between $\varpi$ and $\pi$ as follows:
\begin{df}\label{df:rho}\begin{enumerate}
\item We say $\varphi\colon 2^{[m]}\to\mathbb N$ is decomposable if there exists a partition $I_1\amalg I_2=[m]$ such that $I_1,I_2\neq\emptyset$ and $\varphi(I)=\varphi(I\cap I_1)+\varphi(I\cap I_2)$ for all $I\subset[m]$. 
\item Define a map $\rho\colon Map(2^{[m]},\mathbb N)\to \mathcal P_m,\ \varphi\mapsto \{(I_k,r_k)\}_{k=1}^m$ where
\begin{enumerate}
\item $\varphi(I)=\sum_{k=1}^l \varphi(I\cap I_k)$ for $I\subset [m]$, and $r_k=\varphi(I_k)$, 
\item the restricted map $\left.\varphi\right|_{I_k}$ is indecomposable for each $1\leq k\leq l$. 
\end{enumerate}
\end{enumerate}
\end{df}
\begin{lemma} \label{thm:rho} The map $\rho$ in Definition \ref{df:rho} is well-defined, and $\varpi=\rho\circ \pi$ holds for the maps $\varpi$ and $\pi$ defined in Definitions \ref{df:splitting} and \ref{df:rankmatrix}. 
\end{lemma}
\begin{proof}
First, we show well-definedness of $\rho$. Let $\{I_k^{r_k}\}_{k=1}^l$ and $\{J_k^{s_k}\}_{k=1}^{l'}\in\mathcal P_m$ satisfy the two conditions (i) and (ii) in Definition \ref{df:rho} (2) for $\varphi\colon 2^{[m]}\to\mathbb N$. If $I_k\cap J_{k'}\neq\emptyset$ and $I_k\setminus J_{k'}\neq\emptyset$, then for $I\subset I_k$, we have 
\begin{align*}
\varphi(I)&=\sum_{k''=1}^{l'}\varphi(J_{k''}\cap I)=\varphi(J_{k'}\cap I)+\sum_{k''=1}^{l'}\varphi(J_{k''}\cap I\setminus J_{k'})=\varphi(J_{k'}\cap I)+\varphi(I\setminus J_{k'}) \\
&= \varphi(I_k\cap J_{k'}\cap I)+\varphi(I_k\setminus J_{k'}\cap I)\end{align*}
from (i), and it contradicts to the indecomposability of $\left.\varphi\right|_{I_k}$ in (ii). Hence either $I_k\cap J_{k'}=\emptyset$ or $I_k=J_{k'}$ has to hold. 
Now if $\varpi(v)=\varpi([v_i]_{i=1}^m)=\{I_k^{r_k}\}_{k=1}^l$, we can observe $\pi (v)$ as below.  
\begin{enumerate}
\item From $\langle v_i\rangle_{i=1}^m=\bigoplus_{k=1}^l \langle v_i\rangle_{i\in I_k}$, the following holds for all $I\subset [m]$:
\[\pi(v)(I)=\dim\langle v_i\rangle_{i\in I}=\sum_{k=1}^l \langle v_i\rangle_{i\in I_k\cap I}=\sum_{k=1}^l\pi(v)(I_k\cap I).\] 
\item If $\emptyset\neq J,J'$ and $J\amalg J'=I_k$, then $\langle v_i\rangle_{i\in I_k\cap J}\cap \langle v_i\rangle_{i\in I_k\cap J'}\neq \bm 0$ from the essential indecomposability of $[v_i]_{i\in I_k}$. Hence $\pi(v)|_{I_k}$ is indecomposable from the following inequality: 
\[\pi(v)(I_k)=\dim\langle v_i\rangle_{i\in I_k}< \dim\langle v_i\rangle_{i\in I_k\cap J}+\dim \langle v_i\rangle_{i\in I_k\cap J'}=\pi(v)(I_k\cap J)+\pi(v)(I_k\cap J').\]
\end{enumerate}
Hence we have shown that $\rho\circ\pi=\varpi$. 
\end{proof}
Combining the results in Fact \ref{thm:splittingdecomp} and Lemma \ref{thm:rho}, we obtain the following result. 
\begin{prop} \label{thm:rhobij} 
For the maps $\varpi\colon (\mathbb P^{n-1})^m\overset{\pi}{\to } Map(2^{[m]},\mathbb N)\overset{\rho}{\to}\mathcal P_m$ and the subset $\mathcal P_{n,m}'$ of $\mathrm{Image}(\varpi)$ defined in Definitions \ref{df:splitting}, \ref{df:rankmatrix}, \ref{df:rho}, and Fact \ref{thm:splittingdecomp}, there exist bijections commuting the following diagram. 
\[\xymatrix{
(\mathbb P^{n-1})^m \ar@{.>>}[d] \ar@{>>}[rd]_-{\pi} \ar@{>>}[rrd]^-{\varpi} \\
GL_n\backslash (\mathbb P^{n-1})^m\ar@{.>>}[r]_-{\tilde\pi}& \mathrm{Image}(\pi)  \ar@{.>>}[r]_-{\left.\rho\right|} & \mathrm{Image}(\varpi) \\
GL_n\backslash\varpi^{-1}(\mathcal P'_{n,m}) \ar[r]^-{\simeq}\ar@<-0.3ex>@{^{(}->}[u]  & \left.\rho\right|^{-1}(\mathcal P'_{n,m}) \ar[r]^-{\simeq}\ar@<-0.3ex>@{^{(}->}[u] & \mathcal P'_{n,m} \ar@<-0.3ex>@{^{(}->}[u]
}\]
Furthermore, for an element $\{I_k^{r_k}\}_{k=1}^l\in\mathcal P_{n,m}'$, the corresponding map in $Map(2^{[m]},\mathbb N)$ is given by the following: 
\[\left.\rho\right|^{-1}\left(\{I_k^{r_k}\}_{k=1}^l\right)\colon I\mapsto \sum_{k=1}^l \min\{\#(I\cap I_k),r_k\}. \]
\end{prop}
\begin{proof}
Since $\tilde \varpi=\rho\circ\tilde\pi$ is bijective on the preimage of $\mathcal P_{n,m}'$ from Fact \ref{thm:splittingdecomp}, the restricted maps of $\rho$ and $\tilde \pi$ are also. 
Furthermore, for an element $\{I_k^{r_k}\}_{k=1}^l\in\mathcal P_{n,m}'$, consider a representative $v=[v_i]_{i=1}^m$ defined in Fact \ref{thm:splittingdecomp} (2) of the orbit $\varpi^{-1}\left(\{I_k^{r_k}\}_{k=1}^l\right)$. Then $\{v_i\}_{i\in I_k}$ spans an $r_k$-dimensional subspace and all $r_k$-tuples are linearly independent. Hence we have
\[\pi(v)\colon I\mapsto \dim\langle v_i\rangle_{i\in I}=\dim\bigoplus_{k=1}^l\langle v_i\rangle_{i\in I\cap I_k}=\sum_{k=1}^l\dim\langle v_i\rangle_{i\in I\cap I_k}=\sum_{k=1}^l \min\{\#(I\cap I_k),r_k\}. \]
\end{proof}

\subsection{Classification of configurations via indecomposable splittings}

In this subsection, we focus on the subset $\mathcal P_{4,5}'\subset \mathrm{Image}(\varpi)$ defined in Fact \ref{thm:splittingdecomp}. By definition of $\mathcal P_5$, an element $\{I_k^{r_k}\}_{k=1}^l$ consists of a partition $\{I_k\}_{k=1}^l$ of $[5]$ without any empty sets, and integers $r_k$. Furthermore, it is in the image of $\varpi$ if and only if the the splitting type $\{(\#I_k)^{r_k}\}_{k=1}^l$ satisfies $\sum_{k=1}^lr_k\leq 4$ and $r_k=1$ or $2\leq r_k\leq \#I_k-1$ for each $1\leq k\leq l$. For the condition to be an element of $\mathcal P_{4,5}'$, we shall just alternate the inequality $2\leq r_k\leq \#I_k-1$ with $2\leq r_k= \#I_k-1$. From these observations, we obtain the following classification of splitting types: 

\begin{lemma}\label{thm:splittingtype} For the map $\varpi\colon (\mathbb P^3)^5\to\mathcal P_{5}$ defined in Definition \ref{df:splitting}, 
\begin{enumerate}
\item an element $\{I_k^{r_k}\}_{k=1}^l\in\mathcal P_{5}$ is in the image of $\varpi$ if and only if the the splitting type $\{(\#I_k)^{r_k}\}_{k=1}^l$ is either $\{5^r\}$ ($1\leq r\leq 4$), $\{1^1,4^r\}$ ($1\leq r\leq 3$), $\{2^1,3^r\}$ ($1\leq r\leq 2$), $\{1^1,1^1,3^r\}$ ($1\leq r\leq 2$), $\{1^1,2^1,2^1\}$, or $\{1^1,1^1,1^1,2^1\}$, 
\item an element $\{I_k^{r_k}\}_{k=1}^l\in\mathrm{Image}(\varpi)\subset \mathcal P_{5}$ is in $\mathcal P_{3,5}'$ if and only if the the splitting type $\{(\#I_k)^{r_k}\}_{k=1}^l$ is neither $\{5^3\}$, $\{5^2\}$ nor $\{1^1,4^2\}$. 
\end{enumerate}
\end{lemma}

According to Proposition \ref{thm:rhobij} and Lemma \ref{thm:splittingtype}, we can describe orbits in the fibres on $\mathcal P_{4,5}'$ as follows: 
\begin{thm} \label{thm:split} Consider splittings $\{I_k^{r_k}\}_{k=1}^l\in \mathcal P_{4,5}'\subset \mathcal P_5$ classified in Lemma \ref{thm:splittingtype} and the maps $\varpi\colon (\mathbb P^{n-1})^m\overset{\pi}{\to} Map(2^{[5]},\mathbb N)\overset{\rho}{\to}\mathcal P_5$ defined in Definitions \ref{df:splitting}, \ref{df:rankmatrix}, and \ref{df:rho}. Then the fibres $\varpi^{-1}\left(\{I_k^{r_k}\}_{k=1}^l\right)$ are single $GL_4$-orbit. 

In particular, the one-to-one correspondence between the splittings in $\mathcal P_{4,5}'$, rank matrices in $Map(2^{[5]},\mathbb N)$, and $GL_n$-orbits in $(\mathbb P^{n-1})^m$ is given as in the Table \ref{tab:split} according to the types. 
\end{thm}
\begin{table}[h] 
\centering
\caption{Orbits as the fibres of $\varpi$ on $\mathcal P_{4,5}'$}
\label{tab:split}
\begin{tabular}{c|c|c|c}
splitting type  & rank type & representative of orbit$/\mathcal S_5$ & order of indices\\\hline
$\{5^4\}$ & $(5,10,10)$ & $[e_1,e_2,e_3,e_4,e_1+e_2+e_3+e_4]$ &  \\
$\{5^1\}$ & $(\emptyset)$ & $[e_1,e_1,e_1,e_1.e_1]$  \\
$\{1^1,4^3\}$ & $(5,10,7)$ & $[e_1,e_2,e_3,e_4,e_2+e_3+e_4]$ & $\mathcal S_5/\mathcal S_4$ \\
$\{1^1,4^1\}$ & $(2)$a & $[e_1,e_2,e_2,e_2,e_2]$ &$\mathcal S_5/\mathcal S_4$ \\
$\{2^1,3^2\}$ & $(4,4)$a & $[e_1,e_1,e_2,e_3,e_2+e_3]$ & $\mathcal S_5/(\mathcal S_2\times\mathcal S_3)$ \\
$\{2^1,3^1\}$ & $(2)$b & $[e_1,e_1,e_2,e_2,e_2]$ & $\mathcal S_5/(\mathcal S_2\times\mathcal S_3)$ \\
$\{1^1,1^1,3^2\}$ & $(5,8,5)$ & $[e_1,e_2,e_3,e_4,e_3+e_4]$  & $\mathcal S_5/(\mathcal S_2\times\mathcal S_3)$ \\
$\{1^1,1^1,3^1\}$ & $(3,3)$b & $[e_1,e_2,e_3,e_3,e_3]$ & $\mathcal S_5/(\mathcal S_2\times\mathcal S_3)$ \\
$\{1^1,2^1,2^1\}$  & $(3,3)$a & $[e_1,e_2,e_2,e_3,e_3]$  & $\mathcal S_5/(\mathcal S_2\times\mathcal S_2\times\mathcal S_2)$ \\
$\{1^1,1^1,1^1,2^1\}$ & $(4,6,4)$ & $[e_1,e_2,e_3,e_4,e_4]$ & $\mathcal S_5/(\mathcal S_2\times\mathcal S_3)$
\end{tabular}
\end{table}
To distinguish rank matrices which have the same rank type, but distinct splitting types, we mark on the rank type like as $(3)$a and $(3)$b. 
\begin{proof} From Proposition \ref{thm:rhobij}, we already know that there is a one-to-one correspondence between splittings in $\mathcal P_{4,5}'$, rank matrices, and orbits. In particular, representatives of such orbits are also given in Fact \ref{thm:splittingdecomp}. Hence we shall only check the rank types of corresponding rank matrices. 

From Proposition \ref{thm:rhobij}, the rank matrices corresponding to $\{I_k^{r_k}\}_{k=1}^l\in \mathcal P_{4,5}'$ is given by 
\[\varphi:=\rho|^{-1}\left(\{I_k^{r_k}\}_{k=1}^l\right)\colon I\mapsto \sum_{k=1}^l\min\{\#(I\cap I_k),r_k\}. \]
Hence, the set of maximal elements in $\varphi^{-1}(r)\subset 2^{[5]}$ is bijective to the set 
\[\left\{\{J_k^{s_k}\}_{k=1}^l \left|\ \sum_{k=1}^ls_k=r\ \textrm{where}\ s_k\leq r_k,\ \textrm{and}\ \ \begin{cases}
J_k=I_k & s_k=r_k \\ J_k\subset I_k,\ \#J_k=s_k &s_k<r_k\end{cases}\right.\right\}. \]
Then we obtain the correspondence in Table \ref{tab:split} by a direct computation. 
\end{proof}

\subsection{Parametrisations of infinitely many orbits}

In the previous subsection, we described all orbits obtained as fibres of $\varpi\colon (\mathbb P^3)^5\to\mathcal P_5$ on a subset $\mathcal P_{4,5}'\subset \mathrm{Image}(\varpi)$ determined in Lemma \ref{thm:splittingtype}. Hence we shall only consider the orbit decomposition of each fibre of $\varpi$ on $\mathrm{Image}(\varpi)\setminus\mathcal P_{4,5}'$. Since the splitting types $\{(\#I_k)^{r_k}\}_{k=1}^l$ of elements in $\mathrm{Image}(\varpi)\setminus\mathcal P_{4,5}'$ are fulfilled by $\{5^3\}$, $\{5^4\}$, and $\{1^1,4^2\}$ from Lemma \ref{thm:splittingtype}, we consider the orbit decomposition of each fibre by a case-by-case computation. For a preparation, we define the followings. 
\begin{df} \label{df:pgeneric} We define open dense subsets of projective spaces as follows. 
\begin{align*}
(\mathbb P^2)'&=\left\{[p_1e_1+p_2e_2+p_3e_3]\in\mathbb P^2\left|\ p_1p_2p_3(p_1-p_2)(p_2-p_3)(p_3-p_1)\neq 0\right.\right\} \\
(\mathbb P^1)'&=\left\{[p_1e_1+p_2e_2]\in\mathbb P^1\left|\ p_1p_2(p_1-p_2)\neq 0\right.\right\} \\
(\mathbb P^1\times \mathbb P^1)'&=\left\{[p_1e_1+p_2e_2,q_1e_1+q_2e_2]\in(\mathbb P^1)^2\left|\ p,q\in (\mathbb P^1)',\ p\neq q\right.\right\} 
\end{align*}
\end{df}

\bigskip
\noindent
\underline{{\bf Splitting type $\{5^3\}$}}: Remark that the splitting of type $\{5^3\}$ is only $\{[5]^3\}\in\mathrm{Image}(\varpi)\subset\mathcal P_5$, and we shall only consider its fibre. From Definition \ref{df:splitting}, the notion $\varpi(v)=\varpi([v_i]_{i=1}^5)=\{[5]^3\}$ is equivalent to that $v$ is of rank $3$ and essentially indecomposable. Then we have the following characterisation: 
\begin{lemma}\label{thm:indecomp35} For $v=[v_i]_{i=1}^5\in(\mathbb P^{3})^5$, it is in the fibre $\varpi^{-1}(\{[5]^3\})$ if and only if $v$ is of rank $3$ and there exists a $4$-tuple in $\{v_i\}_{i=1}^5$ which is in a general position (all triples in this $4$-tuple are linearly independent).  
\end{lemma}
\begin{proof}
It is clear that $v$ of rank $3$ is indecomposable if there exists a $4$-tuple in a general position. Hence we shall show the converse. 
If $\varpi(v)=\{[5^3]\}$, then there exists a triple in $\{v_i\}_{i=1}^5$ which forms a basis of $\langle v_i\rangle_{i=1}^5$. We may let it be $\{v_1,v_2,v_3\}$ without loss of generality, and set $v_4=c_1v_1+c_2v_2+c_3v_3$, $v_5=d_1v_1+d_2v_2+c_3v_3$. 

\begin{enumerate}
\item If $c_1c_2c_3\neq 0$, then $\{v_1,v_2,v_3,v_4\}$ is in a general position. 
\item If $c_1=0$ and $c_2c_3\neq 0$, then we can assume $v_4=v_2+v_3$ by normalising. Since $\langle v_1\rangle \cap \langle v_2,v_3,v_4\rangle=\bm 0$, we have $d_1\neq0$ and $d_2v_2+d_3v_3\neq \bm0$ from indecomposability. 

If $d_2d_3\neq 0$, then $\{v_1,v_2,v_3,v_5\}$ is in a general position since $d_1d_2d_3\neq 0$. 

If $d_i=0$ where $\{i,j\}=\{2,3\}$, then $v_5=d_1v_1+d_jv_j=d_1v_1-d_jv_i+d_j(v_i+v_j)=d_1v_1-d_jv_i+d_jv_4$ leads that $\{v_1,v_i,v_4,v_5\}$ is in a general position since $d_1d_j\neq 0$. 

Same argument holds if just one of $c_1,c_2,c_3$ is $0$. 
\item If just two of $c_1,c_2,c_3$ are $0$, then $d_1d_2d_3\neq 0$ holds from indecomposability, because $\langle v_i\rangle_{i=1}^5=\langle v_1\rangle\oplus\langle v_2\rangle\oplus\langle v_3\rangle$. Hence $\{v_1,v_2,v_3,v_5\}$ is in a general position. 
\end{enumerate}
\end{proof}
From Lemma \ref{thm:indecomp35}, we may assume $\{v_1,v_2,v_3,v_4\}$ is in a general position in $3$-dimensional space for $v=[v_i]_{i=1}^5\in \varpi^{-1}(\{[5]^3\})$ without loss of generality. Hence there exists a linear isomorphism $g\in GL_4$ satisfying $g\cdot v=[e_1,e_2,e_3,e_1+e_2+e_3,p]$ where $p=[p_1e_1+p_2e_2+p_3e_3]\in \mathbb P^2$, and it does not depend on the choice of $g\in GL_4$, because an isomorphism stabilising $[e_1,e_2,e_3,e_1+e_2+e_3]$ is a scalar action on $\langle e_1,e_2,e_3\rangle$. Remark that we have the followings: 
\begin{gather*}
|e_1\ e_2\ p|=p_3,\ \ |e_1\ e_3\ p|=-p_2,\ \ |e_2\ e_3\ p|=p_1,\\ |e_1\ e_1+e_2+e_3\ p|=p_3-p_2,\ \ |e_2\ e_1+e_2+e_3\ p|=p_1-p_3, \ \ |e_3\ e_1+e_2+e_3\ p|=p_2-p_1. 
\end{gather*}

\begin{enumerate}
\item If $p_1p_2p_3(p_1-p_2)(p_2-p_3)(p_3-p_1)\neq0$, in other words, if $p\in(\mathbb P^2)'$, then all triples in $\{v_i\}_{i=1}^5$ are linearly independent. Hence the rank matrix of $v$ is $\pi(v)\colon I\mapsto \dim\langle v_i\rangle_{i\in I}=\min\{\#I,3\}$. Define this rank matrix of type $(5,10)$ as $\varphi[5^3]$. Then we have the following. 
\[\begin{array}{ccc}
GL_4\backslash \pi^{-1}(\varphi[5^3]) & \simeq & \left(\mathbb P^2\right)' \\
\rotatebox{90}{$\in$} & & \rotatebox{90}{$\in$} \\
GL_4\cdot[e_1,e_2,e_3,e_1+e_2+e_3,p] & \reflectbox{$\mapsto$} & p
\end{array}\]
\item If exactly one of $p_1,p_2,p_3,p_1-p_2,p_2-p_3,p_3-p_1$ is $0$, then there exists a unique triple $J=\{2,3,5\},\{1,3,5\},\{1,2,5\},\{3,4,5\},\{1,4,5\},\{2,4,5\}$ such that $\{v_i\}_{i\in J}$ is in a general position in $2$-dimensional space, and other triples are all linearly independent respectively. Hence the rank matrix of $v$ is 
\[\pi(v)\colon I\mapsto \dim\langle v_i\rangle_{i\in I}=\begin{cases} \min\{\#I,3\} & I\neq J \\ 2 & I=J, \end{cases}\]
and we define this rank matrix of type $(5,8)$ as $\varphi[5^3;J]$. Considering the case where the $4$-tuple $\{v_i\}_{i\neq j}$ is in a general position for $2\leq j\leq 5$, we also obtain rank matrices $\varphi[5^3;J]$ with $\#J=3$ and $5\notin J$. 

For instance, let $J=\{2,3,5\}$, we have the following: 
\[\begin{array}{ccc}
GL_4\backslash \pi^{-1}(\varphi[5^3;\{2,3,5\}]) & \simeq & \left(\mathbb P^1\right)' \\
\rotatebox{90}{$\in$} & & \rotatebox{90}{$\in$} \\
GL_4\cdot[e_1,e_2,e_3,e_1+e_2+e_3,p_2e_2+p_3e_3] &\mapsto &[p_2e_1+p_3e_2] \end{array}\]
For other $\#J=3$, we can also construct bijections between $GL_4\backslash \pi^{-1}(\varphi[5^3;J])$ and $\left(\mathbb P^1\right)'$ by permuting indices. 
\end{enumerate}
If more than one of $p_1,p_2,p_3,p_1-p_2,p_2-p_3,p_3-p_1$ are $0$. We have $2$ situations. In the following cases, we set $\{1,2,3\}=\{i,j,k\}$. 
\begin{enumerate}
\setcounter{enumi}{2}
\item If $p_i=p_j-p_k=0$, then $p_j=p_k=p_j-p_i=p_k-p_i\neq 0$. Then there exist just $2$ triples $\{v_j,v_k,v_5\}$ and $\{v_i,v_4,v_5\}$ which are in general positions in $2$-dimensional spaces and others are linearly independent. Then the rank matrix of $v$ is 
\[\pi(v)\colon I\mapsto \dim\langle v_i\rangle_{i\in I}=\begin{cases} \min\{\#I,3\} & I\neq \{j,k,5\},\{i,4,5\} \\ 2 & \{j,k,5\},\{i,4,5\}, \end{cases}\]
and we define this rank matrix of type $(5,6)$ as $\varphi[5^3;\{j,k\},\{i,4\}]$. Considering the case where the $4$-tuple $\{v_i\}_{i\neq j}$ is in a general position for $2\leq j\leq 5$, we also obtain rank matrices $\varphi[5^3;J_1,J_2]$ with $\#J_1=\#J_2=2$ and $J_1\amalg J_2=[5]\setminus\{j\}$. 

For instance, if $p_1=p_2-p_3=0$, then $p$ defines a unique element $[e_2+e_3]$ of $\mathbb P^2$, and the fibre $\pi^{-1}\left(\varphi[5^3;\{2,3\},\{1,4\}]\right)$ is a single $GL_4$-orbit through $[e_1,e_2,e_3,e_1+e_2+e_3,e_2+e_3]$. Same argument holds for other $\varphi[5^3;J_1,J_2]$. 
\item If $p_i=p_j=0$, then $p_i-p_j=0$ and $p_k=p_k-p_i=p_k-p_j\neq 0$. Similarly, if $p_i-p_j=p_j-p_k=0$, then $p_k-p_i=0$ and $p_i=p_j=p_k\neq 0$. In these cases, there exists just one pair $\{v_k,v_5\}$ or $\{v_4,v_5\}$ which is proportional and the remaining $4$-tuple $\{v_i\}_{i=1}^4$ is in a general position respectively. By defining $J=\{k,5\}$ or $\{4,5\}$ respectively, we have the rank matrix $\pi(v)\colon I\mapsto \min\{\#(I/J),3\}$ of type $(4,6)$, which we define as $\varphi[5^3,J]$. 
Considering the case where the $4$-tuple $\{v_i\}_{i\neq j}$ is in a general position for $2\leq j\leq 5$, we also obtain rank matrices $\varphi[5^3;J]$ with $\#J=2$ and $5\notin J$. 

For instance, if $p_1=p_2=0$, $p$ defines a unique element $[e_3]$ of $\mathbb P^2$, and the fibre $\pi^{-1}\left(\varphi[5^3;\{3,5\}]\right)$ is a single $GL_4$-orbit through $[e_1,e_2,e_3,e_1+e_2+e_3,e_3]$. Same argument holds for other $\varphi[5^3;J]$ with $\#J=2$. 
\end{enumerate}

\bigskip
\noindent
\underline{{\bf Splitting type $\{5^2\}$:}} Remark that the splitting of type $\{5^2\}$ is only $\{[5]^2\}$. From Definition \ref{df:splitting}, $v=[v_i]_{i=1}^5\in (\mathbb P^3)^5$ is in the fibre $\varpi^{-1}(\{[5]^2\})$ if and only if $v$ is of rank $2$ and essentially indecomposable. 

\begin{lemma}\label{thm:indecomp25} For $v=[v_i]_{i=1}^5\in(\mathbb P^{3})^5$, it is in the fibre $\varpi^{-1}(\{[5]^2\})$ if and only if $v$ is of rank $2$ and there exists a triple in $\{v_i\}_{i=1}^5$ which is in a general position (all pairs are linearly independent).  
\end{lemma}
\begin{proof}
It is clear that $v$ of rank $2$ is indecomposable if there exists a triple in a general position. Hence we shall show the converse. 

If $\varpi(v)=\{[5]^2\}$, then there exists a pair $\{v_i,v_j\}$ which forms a basis of $\langle v_i,v_j,v_k,v_l,v_m\rangle$ where $\{i,j,k,l,m\}=[5]$. If $v_k$, $v_l$ and $v_m$ is proportional to either $v_i$ or $v_j$, then $v$ is decomposed into $\langle  v_i\rangle\oplus \langle v_j\rangle$. Hence at least one of them is proportional to neither $v_i$ nor $v_j$. Hence the claim holds. 
\end{proof}

From Lemma \ref{thm:indecomp25}, we may assume $\{v_1,v_2,v_3\}$ is in a general position in $2$-dimensional space for $v=[v_i]_{i=1}^5\in \varpi^{-1}(\{[5]^2\})$ without loss of generality. Hence there exists a linear isomorphism $g\in GL_4$ satisfying $g\cdot v=[e_1,e_2,e_1+e_2,p,q]$ where $p,q\in \mathbb P^1$, and $(p,q)$ does not depend on the choice of $g\in GL_4$, because an isomorphism stabilising  $[e_1,e_2,e_1+e_2]$ is a scalar action on $\langle e_1,e_2\rangle$. Remark that $p\in(\mathbb P^1)'$ if and only if $p\neq[e_1],[e_2],[e_1+e_2]$. 

\begin{enumerate}
\item If $p,q\in (\mathbb P^1)'$ and $p\neq q$, in other words, if $(p,q)\in \left(\mathbb P^1\times\mathbb P^1\right)'$, then all pairs in $\{v_i\}_{i=1}^5$ are not proportional. Hence the rank matrix of $v$ is $\pi(v)\colon I\mapsto \min\{\#I,2\}$. We define this rank matrix of type $(5)$ as $\varphi[5^2]$, and we have the following bijection: 
\[\begin{array}{ccc}
GL_4\backslash \pi^{-1}(\varphi[5^2]) & \simeq & \left(\mathbb P^1\times\mathbb P^1\right)' \\
\rotatebox{90}{$\in$} & & \rotatebox{90}{$\in$} \\
GL_4\cdot[e_1,e_2,e_1+e_2,p,q] &\mapsto & (p,q). \end{array}\]

\item If exactly one of $p\notin(\mathbb P^1)'$, $q\notin(\mathbb P^1)'$, or $p=q$ holds, then there exist a unique pair $\{v_i\}_{i\in J}$ which is proportional to each others where $J=\{j,4\}$, $\{j,5\}$, or $\{4,5\}$ for $j\in[3]$ respectively. Hence the rank matrix of $v$ is  $I\mapsto \min\{\#(I/J),2\}$, and we define this rank matrix of type $(4)$ as $\varphi[5^2;J]$. Considering the cases where other triples than $\{v_i\}_{i=1}^3$ are in a general position, we also obtain rank matrices $\varphi[5^2;J]$ for $\#J=2$ and $J\subset[3]$. 

For instance, let $J=\{4,5\}$, we have the following bijection: 
\[\begin{array}{ccc}
GL_4\backslash \pi^{-1}(\varphi[5^2;\{4,5\}]) & \simeq & \left(\mathbb P^1\right)' \\
\rotatebox{90}{$\in$} & & \rotatebox{90}{$\in$} \\
GL_4\cdot[e_1,e_2,e_1+e_2,p,p] &\mapsto & p. \end{array}\]
For other $\varphi[5^2;J]$ with $\#J=2$, we can also construct bijections between the fibre $GL_4\backslash \pi^{-1}(\varphi[5^2;J])$ and $(\mathbb P^1)'$ similarly. 

\item If $p,q\notin(\mathbb P^1)'$ and $p\neq q$, then there exist exactly two disjoint pairs $\{v_i,v_4\}$ and $\{v_j,v_5\}$ which are proportional to each others where $1\leq i\neq j\leq 3$. Hence the rank matrix of $v$ is $\pi(v)\colon I\mapsto \min\{\#(I/(i\sim 4,j\sim 5)),2\}$. We define this rank matrix of type $(3)$a as $\varphi[5^2;\{i,4\},\{j,5\}]$. Considering the cases where other triples than $\{v_i\}_{i=1}^3$ are in a general position, we also obtain rank matrices $\varphi[5^2;J_1,J_2]$ for $\#J_1=\#J_2=2$ and $J_1\amalg J_2=[5]\setminus\{4\}$ or $[5]\setminus\{5\}$. 

For instance, $\pi^{-1}\left(\varphi[5^2;\{1,4\},\{2,5\}]\right)$ is a single orbit through $[e_1,e_2,e_1+e_2,e_1,e_2]$. 

\item If $p,q\notin(\mathbb P^1)'$ and $p= q$, then there exists a unique triple $\{v_i,v_4,v_5\}$ which are proportional to each others where $1\leq i\leq 3$. Hence the rank matrix of $v$ is $\pi(v)\colon I\mapsto \min\{\#(I/(i\sim 4\sim 5)),2\}$. We define this rank matrix of type $(3)$b as $\varphi[5^2;\{i,4,5\}]$. Considering the cases where other triples than $\{v_i\}_{i=1}^3$ are in a general position, we also obtain rank matrices $\varphi[5^2;J]$ for $\#J=3$ where $\{4,5\}\nsubseteq J$. 

For instance, $\pi^{-1}\left(\varphi[5^2;\{1,4,5\}]\right)$ is a single orbit through $[e_1,e_2,e_1+e_2,e_1,e_1]$. 
\end{enumerate}

\bigskip
\noindent
\underline{{\bf Splitting type $\{1^1,4^2\}$: }} Remark that the splitting of type $\{1^1,4^2\}$ is fulfilled with $\{\{i\}^1,[5]\setminus\{i\}^2\}$ for some $i\in [5]$. In this argument, we focus only on $\{\{1\}^1,\{2,3,4,5\}^2\}$, and the others can be covered by the action of $\mathcal S_5/\mathcal S_4$. From Definition \ref{df:splitting}, $v=[v_i]_{i=1}^5\in (\mathbb P^3)^5$ is in the fibre $\varpi^{-1}(\{\{1\}^1,\{2,3,4,5\}^4\})$ if and only if $[v_i]_{i=2}^5$ is of rank $2$ and essentially indecomposable, and $v_1$ is not in $\langle v_i\rangle_{i=2}^5$. 
Furthermore, we can prove that $[v_i]_{i=2}^5$ of rank $2$ is essentially indecomposable if and only if there exists a triple in $\{v_i\}_{i=2}^5$ which is in a general position by a similar argument as in Lemma \ref{thm:indecomp25}. 

Hence, we may assume $\{v_2,v_3,v_4\}$ is in a general position in $2$-dimensional space for $v=[v_i]_{i=1}^5\in \varpi^{-1}(\{\{1\}^1,\{2,3,4,5\}^2\})$ without loss of generality. Hence there exists a linear isomorphism $g\in GL_4$ satisfying $g\cdot v=[e_3,e_1,e_2,e_1+e_2,p]$ where $p=[p_1e_1+p_2e_2]\in \mathbb P^1$, which does not depend on the choice of $g\in GL_4$, because an isomorphism stabilising $[e_1,e_2,e_1+e_2]$ is a scalar action on $\langle e_1,e_2\rangle$. 

\begin{enumerate}
\item If $p\in(\mathbb P^1)'$, then all pairs of $\{v_i\}_{i=2}^5$ are not proportional. Hence the rank matrix of $v$ is $\pi(v)\colon I\mapsto \#(I\cap\{1\})+\min\{\#(I\setminus\{1\}),2\}$. We define this rank matrix of type $(5,5)$ as $\varphi[4^2;1]$, and we have the following bijection: 
\[\begin{array}{ccc}
GL_4\backslash \pi^{-1}(\varphi[4^2;1]) & \simeq & \left(\mathbb P^1\right)' \\
\rotatebox{90}{$\in$} & & \rotatebox{90}{$\in$} \\
GL_4\cdot[e_3,e_1,e_2,e_1+e_2,p] &\mapsto & p \end{array}\]

\item If $p\notin(\mathbb P^1)'$, then there exist a unique pair $\{v_i,v_5\}$ which is proportional where $2\leq i\leq 4$. Hence the rank matrix of $v$ is  $I\mapsto \#(I\cap\{1\})+\min\{\#((I\setminus\{1\})/i\sim 5),2\}$, and we define this rank matrix of type $(4,4)$b as $\varphi[4^2;1;\{i,5\}]$. Considering the cases where other triples than $\{v_i\}_{i=2}^4$ are in a general position, we also obtain rank matrices $\varphi[4^2;1;J]$ for $\#J=2$ and $J\subset \{2,3,4,5\}$ with $5\notin J$. 

For instance, $\pi^{-1}\left(\varphi[4^2;1;\{4,5\}]\right)$ is a single orbit through $[e_3,e_1,e_2,e_1+e_2,e_1+e_2]$. 
\end{enumerate}

Combining these arguments, we have the following result: 
\begin{thm} \label{thm:parameter}
 Consider splittings $\{I_k^{r_k}\}_{k=1}^l\in \mathrm{Image}(\varpi)\setminus\mathcal P_{4,5}'$ classified in Lemma \ref{thm:splittingtype} and the maps $\varpi\colon (\mathbb P^{n-1})^m\overset{\pi}{\to} Map(2^{[5]},\mathbb N)\overset{\rho}{\to}\mathcal P_5$ defined in Definitions \ref{df:splitting}, \ref{df:rankmatrix}, and \ref{df:rho}. 
 \begin{enumerate}
 \item The fibres $\left.\rho\right|^{-1}\left(\{I_k^{r_k}\}_{k=1}^l\right)\subset Map(2^{[5]},\mathbb N)$ are fulfilled by the rank matrices $\varphi[\ast]\in Map(2^{[5]},\mathbb N)$ listed in Table \ref{tab:parameter1}, under the notation $\{i,j,k,l,m\}=[5]$. 
\item For each rank matrix $\varphi[\ast]$ listed in Table \ref{tab:parameter1}, the fibre $\pi^{-1}(\varphi[\ast])$ is decomposed into $GL_4$-orbits which have faithful parameters listed in Table \ref{tab:parameter2}. 
\end{enumerate}
\end{thm}

\begin{table}[h]
\centering
\caption{Rank matrices in the fibres on $\mathrm{Image}(\varpi)\setminus \mathcal P_{4,5}'$}
\label{tab:parameter1}
\begin{tabular}{c|c|c|c}
splitting & rank matrix & type & permutation \\
\hline\hline 
\multirow{9}{*}{$\{5^3\}$}
& $\varphi[5^3]\colon I\mapsto \min\{\#I,3\}$ & $(5,10)$ & \\
\cline{2-4}
& \begin{tabular}{c}$\varphi[5^3;J]\colon I\mapsto\begin{cases} \min\{\#I,3\} & I\neq \{i,j,k\} \\ 2 & I=\{i,j,k\}\end{cases}$ \\ where $\#J=3$\end{tabular} & $(5,8)$ & $\mathcal S_5/(\mathcal S_2\times\mathcal S_3)$ \\
\cline{2-4}
& \begin{tabular}{c}$\varphi[5^3;J_1,J_2]\colon I\mapsto\begin{cases} \min\{\#I,3\} & I\neq [5]\setminus J_\ast \\ 2 & I=[5]\setminus J_\ast\end{cases}$ \\ where $\#J_1=\#J_2=2$, $J_1\cap J_2=\emptyset$ \end{tabular}  & $(5,6)$ & $\mathcal S_5/(\mathcal S_2)^3$ \\
\cline{2-4}
& \begin{tabular}{c} $\varphi[5^3;J]\colon I\mapsto \min\{\#(I/J),3\}$ \\ where $\#J=2$\end{tabular} & $(4,6)$ & $\mathcal S_5/(\mathcal S_2\times\mathcal S_3)$\\
\hline
\multirow{7}{*}{$\{5^2\}$}
& $\varphi[5^2]\colon I\mapsto \min\{\#I,2\}$ & $(5)$ & \\
\cline{2-4}
& \begin{tabular}{c}$\varphi[5^2;J]\colon I\mapsto\min\{\#J,2\}$ \\ where $\#J=2$\end{tabular} & $(4)$ & $\mathcal S_5/(\mathcal S_2\times\mathcal S_3)$ \\
\cline{2-4}
& \begin{tabular}{c}$\varphi[5^2;J_1,J_2]\colon  I\mapsto\min\{\#(I/J_1/J_2,2\}$ \\ where $\#J_1=\#J_2=2$ and $J_1\cap J_2=\emptyset$\end{tabular}  & $(3)$a & $\mathcal S_5/(\mathcal S_2)^3$ \\
\cline{2-4}
& \begin{tabular}{c} $\varphi[5^3;J]\colon I\mapsto \min\{\#(I/J),2\}$ \\ where $\#J=3$ \end{tabular} & $(3)$b & $\mathcal S_5/(\mathcal S_2\times\mathcal S_3)$\\
\hline
\multirow{3}{*}{$\{1^1,4^2\}$}
& $\varphi[4^2;i]\colon I\mapsto \#(I\cap\{i\})+\min\{\#(I\setminus\{i\}),2\}$ & $(5,5)$ & $\mathcal S_5/\mathcal S_4$ \\
\cline{2-4}
& \begin{tabular}{c}$\varphi[4^2;i;J]\colon I\mapsto \#(I\cap \{i\})+\min\{\#((I\setminus\{i\})/J),2\}$ \\ where $\#J=2$ and $i\notin J$\end{tabular} & $(4,4)$b & $\mathcal S_5/(\mathcal S_2\times\mathcal S_2)$ 
\end{tabular}
\end{table}

\begin{table}[H]
\centering
\caption{Orbits in the fibres on $\mathrm{Image}(\varpi)\setminus \mathcal P_{4,5}'$}
\label{tab:parameter2}
\begin{tabular}{c|c|c|c|c|c}
rank matrix &  orbits & parameter  \\
\hline\hline 
$\varphi[5^3]$ & $GL_4\cdot [e_1,e_2,e_3,e_1+e_2+e_3,p]$ & $p\in (\mathbb P^2)'$  \\
$\varphi[5^3;\{1,2,5\}]$ & $GL_4\cdot [e_1,e_2,e_3,e_1+e_2+e_3,p]$ & $p\in (\mathbb P^1)'$  \\
$\varphi[5^3;\{1,2\},\{3,4\}]$ & $GL_4\cdot [e_1,e_2,e_3,e_1+e_2+e_3,e_1+e_2]$ & single  \\
$\varphi[5^3,\{1,5\}]$ & $GL_4\cdot [e_1,e_2,e_3,e_1+e_2+e_3,e_1]$ & single \\\hline
$\varphi[5^2]$ & $GL_4\cdot [e_1,e_2,e_1+e_2,p,q]$ & $(p,q)\in (\mathbb P^1\times\mathbb P^1)'$  \\
$\varphi[5^2;\{4,5\}]$ & $GL_4\cdot [e_1,e_2,e_1+e_2,p,p]$ & $p\in (\mathbb P^1)'$  \\
$\varphi[5^2;\{1,4\},\{2,5\}]$ & $GL_4\cdot [e_1,e_2,e_1+e_2,e_1,e_2]$ & single  \\
$\varphi[5^2;\{1,4,5\}]$ & $GL_4\cdot [e_1,e_2,e_1+e_2,e_1,e_1]$ & single  \\\hline
$\varphi[4^2;1]$ & $GL_4\cdot [e_3,e_1,e_2,e_1+e_2,p]$ & $p\in (\mathbb P^1)'$  \\
$\varphi[4^2;1;\{2,5\}]$ & $GL_4\cdot [e_3,e_1,e_2,e_1+e_2,e_1]$ & single
\end{tabular}
\end{table}

\section{Closure relations of orbits}

In this section, we determine closures of all orbits in terms of the classification of orbits in Theorems \ref{thm:split} and \ref{thm:parameter}. 

\subsection{Orbits as fibres of splittings\label{clsplit}}

First of all, we focus on the orbits which are obtained as fibres of the $GL_n$-invariant map $\varpi\colon (\mathbb P^{n-1})^m\to \mathcal P_m$, which are classified in Theorem \ref{thm:split} for the case where $(n,m)=(4,5)$. 

\bigskip
Now, for the $GL_n$-invariant map $\pi\colon (\mathbb P^{n-1})^m\to Map(2^{[m]},\mathbb N)$ defined in Definition \ref{df:rankmatrix}, fix an element $\varphi\in Map(2^{[m]},\mathbb N)$. Then we have the following characterisation of the $GL_n$-invariant fibre $\pi^{-1}(\varphi)\subset(\mathbb P^{n-1})^m$: 
\begin{align*}
\pi^{-1}(\varphi)&=\left\{[v_i]_{i=1}^m\in (\mathbb P^{n-1})^m \left|\ \dim\langle v_i\rangle_{i\in I}=\varphi(I)\ \textrm{for each}\ I\subset[m]\right. \right\} \\
&=\left\{[v_i]_{i=1}^m\in (\mathbb P^{n-1})^m \left|\ \begin{array}{l}\textrm{all $(\varphi(I)+1)$-minors of $(v_i)_{i\in I}$ vanish, and there exists} \\ \textrm{some non-vanishing $\varphi(I)$-minor of $(v_i)_{i\in I}$ for each $I\subset [m]$}\end{array}\right. \right\}.
\end{align*}
Hence, by denoting $\psi\leq\varphi$ if $\psi(I)\leq \varphi(I)$ for all $I\subset[m]$, we have 
\begin{align*}
\overline{\pi^{-1}(\varphi)}&\subset \left\{[v_i]_{i=1}^m\in (\mathbb P^{n-1})^m \left|\textrm{ all $(\varphi(I)+1)$-minors of $(v_i)_{i\in I}$ vanish for each $I\subset[m]$}\right. \right\} \\
&=\left\{[v_i]_{i=1}^m\in (\mathbb P^{n-1})^m \left|\ \dim\langle v_i\rangle_{i\in I}\leq \varphi(I)\ \textrm{for}\ I\subset[m]\right. \right\} \\
&=\coprod_{\psi\leq\varphi}\pi^{-1}(\psi), 
\end{align*} 
and the fibre $\pi^{-1}(\varphi)$ is open in it. Remark that, if the fibre $\pi^{-1}(\varphi)$ is not empty and the closed subset $\coprod_{\psi\leq \varphi}\pi^{-1}(\psi)$ is irreducible, then the equality holds. 
As we saw in Proposition \ref{thm:rhobij}, there are bijections for a subset $\mathcal P_{n,m}'\subset\mathrm{Image}(\varpi)$ as below: 
\[\begin{array}{cccccc}
\tilde\varpi\colon & GL_n\backslash\varpi^{-1}(\mathcal P_{n,m}') & \overset{\tilde\pi}{\simeq} & \left.\rho\right|^{-1}(\mathcal P_{n,m}') & \overset{\left.\rho\right|}{\simeq} & \mathcal P_{n,m}' \\
& \rotatebox{90}{$\in$} && \rotatebox{90}{$\in$} &&\rotatebox{90}{$\in$} \\
& GL_n\cdot[v_i]_{i=1}^m &\mapsto & \left(I\mapsto \dim\langle v_i\rangle_{i\in I}\right)  \\
&&& \left(I\mapsto \sum_{k=1}^l\min\{\#(I\cap I_k),r_k\}\right) & \reflectbox{$\mapsto$} & \{I_k^{r_k})\}_{k=1}^l
\end{array}\]
For an element $\mathcal I=\{I_k^{r_k}\}_{k=1}^l\in\mathcal \mathcal P_{n,m}'$, and $\varphi_{\mathcal I}:=\left.\rho\right|^{-1}(\mathcal I)$, then the fibre $\varpi^{-1}(\mathcal I)=\pi^{-1}(\varphi_{\mathcal I})\subset (\mathbb P^{n-1})^m$ is a single $GL_n$-orbit, which is classified in Theorem \ref{thm:split} where $(n,m)=(4,5)$. 
Furthermore, consider the condition $\dim\langle v_i\rangle_{i\in I}\leq \varphi_{\mathcal I}(I)$ for all $I\subset[m]$. If it holds for $[v_i]_{i=1}^m\in(\mathbb P^{n-1})^m$, then $\dim\langle v_i\rangle_{i\in I_k}\leq \varphi_{\mathcal I}(I_k)=r_k$. Conversely, if $\dim\langle v_i\rangle_{i\in I_k}\leq r_k$ for all $k\in[l]$, then we have $\dim\langle v_i\rangle_{i\in I}\leq\sum_{k=1}^l\dim\langle x_i\rangle_{i\in I\cap I_k}\leq \sum_{k=1}^l\min\{\#(I\cap I_k),r_k\}=\varphi_{\mathcal I}(I)$. Hence we can interpret the set $\coprod_{\varphi\leq \varphi_{\mathcal I}}\pi^{-1}(\varphi)$ as below: 
\begin{align*}
\coprod_{\varphi\leq\varphi_{\mathcal I}}\pi^{-1}(\varphi)&=\left\{[v_i]_{i=1}^m\in (\mathbb P^{n-1})^m \left|\ \dim\langle v_i\rangle_{i\in I}\leq \varphi_{\mathcal I}(I)\ \textrm{for each}\ I\subset[m]\right. \right\} \\
&=\left\{[v_i]_{i=1}^m\in (\mathbb P^{n-1})^m \left|\ \dim\langle v_i\rangle_{i\in I_k}\leq r_k\ \textrm{for each}\ 1\leq k\leq l\right. \right\} \\
&\simeq \prod_{k=1}^l X_{n,\#I_k}(r_k)
\end{align*}
where $X_{n,m}(r)$ denotes the irreducible closed subvariety of $(\mathbb P^{n-1})^m$ consisting of all elements with rank less than or equal to $r$. Hence combining these results, we have the following: 
\begin{df}\label{df:leq}
For $\varphi,\psi\in Map(2^{[m]},\mathbb N)$, we say $\psi\leq\varphi$ if $\psi(I)\leq \varphi(I)$ for all $I\subset [m]$. 
\end{df}
\begin{prop}\label{thm:clsplit} 
Consider the $GL_n$-invariant maps $\varpi\colon (\mathbb P^{n-1})^m\overset{\pi}{\to} Map(2^{[m]},\mathbb N) \overset{\rho}{\to} \mathcal P_m$ and the subset $\mathcal P_{n,m}'\subset \mathrm{Image}(\varpi)$, on which the fibres of $\varpi$ are single $GL_n$-orbits. For a rank matrix $\varphi\in \rho|^{-1}(\mathcal P_{n,m}')\subset Map(2^{[m]},\mathbb N)$, the closure of the orbit $\pi^{-1}(\varphi)$ is given as below: 
\begin{equation} \label{eq:clsingle}
\overline{\pi^{-1}(\varphi)}=\coprod_{\psi\leq\varphi}\pi^{-1}(\psi).\end{equation}
In particular, for the $GL_4$-orbits classified in Theorem \ref{thm:split} 
, the equality (\ref{eq:clsingle}) holds. 
\end{prop}

\subsection{Infinitely many orbits in the fibres of rank matrices\label{clparameter}}

In this subsection, we determine the closures of orbits which are \emph{not} obtained as fibres of $\varpi\colon (\mathbb P^3)^5\to\mathcal P_5$, but are decomposed into several fibres of $\pi\colon (\mathbb P^3)^5\to Map(2^{[5]},\mathbb N)$, and are parametrised by some infinite parameters. These orbits lie in the fibres on $\mathrm{Image}(\varpi)\setminus\mathcal P_{4,5}'$, and are classified in Theorem \ref{thm:parameter}. 

Remark that there exist closed embeddings $\tilde\iota\colon GL_r\backslash (\mathbb P^{r-1})^m \hookrightarrow GL_n\backslash (\mathbb P^{n-1})^m,\ GL_r\cdot[v_i]_{i=1}^m\mapsto GL_n\cdot[\iota(v_i)]_{i=1}^m$ for linear inclusions $\iota\colon \mathbb K^r\hookrightarrow \mathbb K^n$, which preserve the splittings and rank matrices. Hence for $v \in (\mathbb P^{r-1})^m$, we have $\tilde\pi(\tilde\iota(GL_r\cdot v))=\tilde\pi(GL_r\cdot v)$ (resp. $\tilde\varpi$), and $\overline{GL_n\cdot \iota(v)}=\tilde\iota\left(\overline{GL_r\cdot v}\right)$. 

\subsubsection{Orbits of the splitting type $\{5^2\}$}

For the unique splitting $\{[5]^2\}\in \mathrm{Image}(\varpi)\setminus\mathcal P_{4,5}'$ of type $\{5^2\}$, according to Theorem \ref{thm:parameter}, there are $4$ types, $(5)$, $(4)$, $(3)$a and $(3)$b of rank matrices. 

\bigskip
\noindent
\bigskip
\underline{{\bf Rank type $(5)$: }} 

Consider the rank matrix $\varphi[5^2]\colon I\mapsto \min\{\#I,2\}$ of type $(5)$. Then an orbit with this rank matrix is represented as $GL\cdot[e_1,e_2,e_1+e_2,p,q]$ where $(p,q)\in\left(\mathbb P^1\times\mathbb P^1\right)'$. 

From now on, for the polynomial ring on $(2\times 5)$-matrices $X$, the notion $|i,j|$ denotes the polynomial defined by the $2$-minor of $X$ consisting of the $i$-th and $j$-th columns, which is homogeneous for each columns. Then we define a homogeneous ideal $I_{p,q}$ generated by the following $5$ polynomials. 
\begin{align}
P_1(X)&=q_1(p_2-p_1)|2,4||5,3|+p_1(q_2-q_1)|2,5||3,4| \notag\\
P_2(X)&=q_2(p_2-p_1)|1,4||5,3|+p_2(q_2-q_1)|1,5||3,4| \notag \\
P_3(X)&=p_1q_2|1,4||5,2|+p_2q_1|1,5||2,4| \label{eq:52} \\
P_4(X)&=q_2|1,3||5,2|+q_1|1,5||2,3| \notag\\
P_5(X)&=p_2|1,3||4,2|+p_1|1,4||2,3| \notag
\end{align}
Since these polynomials are relatively invariant under the $GL_2$-action, the ideal $I_{p,q}$ defines a $GL_2$-stable closed subset $Z(I_{p,q})$ of $(\mathbb P^1)^5$. 
Clearly we have $v_{p,q}:=[e_1,e_2,e_1+e_2,p,q]\in Z(I_{p,q})$, hence $\overline{GL_2\cdot v_{p,q}}\subset Z(I_{p,q})$. 
We shall show the converse inclusion. For this purpose, at first we determine the set $Z(I_{p,q})$ according to the classification of orbits. 
\begin{enumerate}
\item First, we consider the case where all $2$-minors of $v=[v_i]_{i=1}^5$ do not vanish, or equivalently, the rank matrix of $v$ is $\varphi[5^2]$. In this case, $v$ is in an orbit through $v_{r,s}$ for some $(r,s)\in\left(\mathbb P^1\times\mathbb P^1\right)'$. Then we have $P_4(v)=q_2s_1-q_1s_2$ and $P_5(v)=p_2r_1-p_1r_2$, which leads that $v_{r,s}$ is in $Z(I_{p,q})$ if and only if $(p,q)=(r,s)$. 
Hence in this case, $GL_4\cdot v_{p,q}$ with the rank matrix $\varphi[5^2]$ of type $(5)$ is the only orbit in $Z(I_{p,q})$. 

\item On the other hand, if there exists a vanishing $2$-minor $|i,j|$ of $v$, then 
\begin{align*}
P_k(v)=|i,l||i,m|,\ P_l(v)=|i,m||i,k|,\ P_m(v)=|i,k||i,l|
\end{align*}
up to some non-zero scalar where $\{i,j,k,l,m\}=[5]$. Hence, if $v\in Z(I_{p,q})$, then at least two of $[v_k],[v_l],[v_m]$ have to coincide with $[v_i]=[v_j]$. In other words, at least $4$ columns have to be proportional to each others. 
Conversely, if $4$ columns of $v$ are proportional to each other, then all polynomials vanish because each term of them can be divided by some $2$-minor consisting of $2$ of those $4$ columns. 
Hence in this case, the orbits contained in $Z(I_{p,q})$ is fulfilled with the the splitting types $\{1^1,4^1\}$ and $\{5^1\}$, which correspond to the rank types $(2)$a and $(1)$. 
\end{enumerate}
Now, we shall show that $Z(I_{p,q})$ is contained in the closure of $GL_2\cdot v_{p,q}$. Consider the map 
\[\begin{array}{ccc}
\left(\mathbb P^1\times\mathbb P^1\times\mathbb P^1\right)' &\to &(\mathbb P^1)^5 \\
\rotatebox{90}{$\in$} && \rotatebox{90}{$\in$} \\ 
v=[v_i]_{i=1}^3 &\mapsto &[v_1,v_2,v_3,p_1|3,2|v_1+p_2|1,3|v_2, q_1|3,2|v_1+q_2|1,3|v_2]
\end{array}\]
where $\left(\mathbb P^1\times\mathbb P^1\times\mathbb P^1\right)'$ denotes the open dense subset of $(\mathbb P^1)^3$ consisting of all elements of the rank $2$. Then the preimage of $GL_2\cdot v_{p,q}$ consists of all elements such that $v_1,v_2,v_3$ are in a general position, which is an open dense subset of $\left(\mathbb P^1\times\mathbb P^1\times\mathbb P^1\right)'$. Hence whole image of this map is contained in the closure of $GL_2\cdot v_{p,q}$. 

Now, if we consider $v=[v_1,v_2,v_3]\in (\mathbb P^1\times \mathbb P^1\times \mathbb P^1)'$ where $[v_i]=[v_j]\neq [v_k]$ for $\{i,j,k\}=[3]$, then the value of this map corresponds to the orbit consisting of all elements satisfying $[v_i]=[v_j]=[v_4]=[v_5]\neq [v_k]$, whose splitting is $\{\{k\}^1,[5]\setminus\{k\}^1\}$ of type $\{1^1,4^1\}$. By considering the map permuting indices, we also obtain the orbits with the splittings $\{\{k\}^1,[5]\setminus\{k\}^1\}$ where $k=4,5$ as the values of the map. 

Hence we have shown that the orbits of the splitting type $\{1^1,4^1\}$ are all contained in $\overline{GL_2\cdot v_{p,q}}$. Furthermore, the orbit with the splitting type $\{5^1\}$ is contained in the closure of those of type $\{1^1,4^1\}\in \mathcal P_{2,4}'$ from Proposition \ref{thm:clsplit}, hence we have the following result: 

\begin{lemma} \label{thm:cl52} For $v_{p,q}=[e_1,e_2,e_1+e_2,p,q]\in(\mathbb P^3)^5$ with the rank matrix $\varphi[5^2]$ of type $(5)$ where $(p,q)\in (\mathbb P^1\times\mathbb P^1)'$, the closure of the orbit through $v_{p,q}$ is the zero locus of the ideal generated by the polynomials in (\ref{eq:52}), and 
\[\overline{GL_4\cdot v_{p,q}}=GL_4\cdot v_{p,q}\ \amalg\ \coprod_{i=1}^5\varpi^{-1}\left(\{\{i\}^1,[5]\setminus\{i\}^1\}\right)\ \amalg\ \varpi^{-1}\left(\{[5]^1\}\right). \]
\end{lemma}

\noindent
\underline{{\bf Rank type $(4)$: }}

\bigskip
Consider the rank matrix $\varphi[5^2;J]\colon I\mapsto \min\{\#(I/J),2\}$ of type $(4)$ where $\#J=2$. It is clear that the closed subset $\{v=[v_i]_{i=1}^5\left|\ \dim\langle v_i\rangle_{i\in J}=1\right.\}\subset (\mathbb P^1)^5$ is naturally identified with $(\mathbb P^1)^4$. Hence we shall only consider the $GL_2$-orbit through $v_p:=[e_1,e_2,e_1+e_2,p]$ for $p=[p_1e_1+p_2e_2]\in (\mathbb P^1)'$. 
Now, let us consider the polynomial 
\begin{align}
P(X)&=p_2|1,3||4,2|+p_1|1,4||2,3| \label{eq:522} \\
&=-p_2|1,2||3,4|+(p_1-p_2)|1,4||2,3|=(p_2-p_1)|1,3||4,2|-p_1|1,2||3,4|,\notag \end{align}
and let $Z(I_p)\subset(\mathbb P^1)^4$ be the zero locus of the principle ideal $I_p$ generated by the polynomial (\ref{eq:522}). Then we clearly have $\overline{GL_2\cdot v_p}\subset Z(I_p)$. Hence we shall determine the set $Z(I_p)$ and show the converse inclusion. 

\begin{enumerate}
\item First, we consider the case where all $2$-minors of $v=[v_i]_{i=1}^4\in(\mathbb P^1)^4$ do not vanish. In this case, $v$ is in an orbit $GL_2\cdot v_q$ for some $q\in\left(\mathbb P^1\right)'$, and we have $P(v)=p_2q_1-p_1q_2$. 
Hence in this case, $GL_2\cdot v_p$ is the only orbit contained in $Z(I_p)$. 

\item On the other hand, assume that there exists a vanishing $2$-minor $|i,j|$ of $v$. If $P(v)=0$, then $|i,k||j,l|=|i,l||j,k|=0$ where $\{i,j,k,l\}=[4]$. Hence either $[v_k]$ or $[v_l]$ has to coincides with $[v_i]=[v_j]$. In other words, at least $3$ columns have to be proportional to each others. 
Conversely, if $3$ columns of $v$ are proportional to each others, then the polynomial $P$ vanishes because each term of it can be divided by some $2$-minor consisting of $2$ of those $3$ columns. 
Hence in this case, orbits in $Z(I_p)$ is fulfilled with the splitting types $\{1^1,3^1\}$ and $\{4^1\}$. 
\end{enumerate}
Now, we shall show that $Z(I_p)$ is contained in the closure of $GL_2\cdot v_p$. Consider the map 
\[\begin{array}{ccc}
\left(\mathbb P^1\times\mathbb P^1\times\mathbb P^1\right)' &\to &(\mathbb P^1)^5 \\
\rotatebox{90}{$\in$} && \rotatebox{90}{$\in$} \\
v=[v_i]_{i=1}^3 &\mapsto &[e_1,e_2,e_3,p_1|3,2|v_1+p_2|1,3|v_2]
\end{array}\]
where $\left(\mathbb P^1\times\mathbb P^1\times\mathbb P^1\right)'$ is the open dense subset of $(\mathbb P^1)^3$ as in the previous case. Similarly to the previous case, whole image of this map is contained in the closure of $GL_2\cdot v_p$. 

If we consider a rank $2$ matrix such that $[v_i]\neq [v_j]= [v_k]$ where $\{i,j,k\}=[3]$, then the value of this map corresponds to the orbit satisfying $[v_i]\neq [v_j]=[v_k]=[v_4]$ whose splitting is $\{\{i\}^1,[4]\setminus\{i\}^1\}\in \mathcal P_{2,4}'$ of the type $\{1^1,3^1\}$. By considering the map permuting indices, we also obtain the orbit with the splitting $\{\{4\}^1,[3]^1\}$ as the values of the map. 

Hence we have shown that the orbits of the splitting type $\{1^1,3^1\}$ are all contained in $\overline{GL_2\cdot v_{p}}$. Furthermore, the orbit with the splitting type $\{4^1\}$ is contained in the closure of those of type $\{1^1,3^1\}\in \mathcal P_{2,4}'$ from Proposition \ref{thm:clsplit}, hence we have the following result: 

\begin{lemma} \label{thm:cl522} 
\begin{enumerate}
\item For $v_{p}=[e_1,e_2,e_1+e_2,p]\in(\mathbb P^3)^4$ where $p\in (\mathbb P^1)'$, the closure of the orbit through $v_p$ is the zero locus of the ideal generated by the polynomial (\ref{eq:522}), and 
\[\overline{GL_4\cdot v_{p}}=GL_4\cdot v_{p}\ \amalg\ \coprod_{i=1}^4\varpi^{-1}\left(\{\{i\}^1,[4]\setminus\{i\}^1\}\right)\ \amalg\ \varpi^{-1}\left(\{[4]^1\}\right). \]
\item In particular, for an orbit $\mathcal O\subset (\mathbb P^3)^5$ with the rank matrix $\varphi[5^2;J]$ of type $(4)$ where $p\in (\mathbb P^1)'$ and $\#J=2$, 
\[\overline{\mathcal O}=\mathcal O\ \amalg\ \varpi^{-1}\left(\{J^1,[5]\setminus J^1\}\right)\ \amalg\ \coprod_{i\notin J}\varpi^{-1}\left(\{\{i\}^1,[5]\setminus\{i\}^1\}\right)\ \amalg\ \varpi^{-1}\left(\{[5]^1\}\right). \]
\end{enumerate}
\end{lemma}

\noindent
\underline{{\bf Rank type $(3)$: }} 

\bigskip
The last rank matrices with the splitting $\{[5]^2\}$ are of types $(3)$a or $(3)$b. In both cases, the fibre of a rank matrix is a single orbit, and an element $v=[v_i]_{i=1}^5$ of these types has $3$ distinct columns up to constant, and they are in a general position. Hence we only have to consider the closure of the $GL_2$-orbit through $[e_1,e_2,e_1+e_2]\in (\mathbb P^1)^3$, whose splitting is $\{[3]^2\}\in\mathcal P_{2,3}'$. 
Then from Proposition \ref{thm:clsplit}, we have the following result: 
\begin{lemma} \label{thm:cl523} 
For a rank matrix $\varphi$ of type $(3)$, the closure of the orbit $\pi^{-1}(\varphi)$ is $\coprod_{\psi\leq \varphi} \pi^{-1}(\psi)$. 
\end{lemma}

\subsubsection{Orbits of the splitting type $\{1^1,4^2\}$}

From Definition \ref{df:splitting}, splittings of type $\{1^1,4^2\}$ are fulfilled with $\{\{j\}^1,[5]\setminus\{j\}^2\}$ where $j\in[5]$. Furthermore, each splitting corresponds to $2$ types of rank matrices $\varphi[4^2;j]$ of type $(5,5)$, and $\varphi[4^2;j;J]$ of type $(4,4)$b where $j\notin J$ and $\#J=2$, which are listed in Table \ref{tab:parameter1}.  

Now, fix an orbit $GL_4\cdot [v_i]_{i=1}^5\subset (\mathbb P^3)^5$ with the splitting $\{\{j\}^1,[5]\setminus\{j\}^2\}$. Then, the splitting of $[v_i]_{i\neq j}\in(\mathbb P^3)^4$ is $\{[4]^2\}$, and we have 
\[GL_4\cdot [v_i]_{i=1}^5=\left\{[w_i]_{i=1}^5\in(\mathbb P^3)^5\left|\ [w_i]_{i\neq j}\in GL_4\cdot [v_i]_{i\neq j}\ \textrm{and}\ w_j\notin \langle w_i\rangle_{i\neq j}\right.\right\}. \]
Hence, the orbit $GL_4\cdot [v_i]_{i=1}^5$ is open in the irreducible closed set $\overline{GL_4\cdot [v_i]_{i\neq j}}\times \mathbb P^3$. Furthermore, the closure of the orbit $GL_4\cdot [v_i]_{i\neq j}$ of the splitting $\{[4]^2\}$ is already determined in Lemmas \ref{thm:cl522} and \ref{thm:cl523}, hence we have the following result: 
\begin{lemma} \label{thm:cl42} For the orbit $GL_4\cdot [v_i]_{i=1}^5$ with the splitting $\{[5]\setminus\{j\}^2,\{j\}^1\}$ where $j\in [5]$, we have $\overline{GL_4\cdot [v_i]_{i=1}^5}=\overline{GL_4\cdot [v_i]_{i\neq j}}\times \mathbb P^3$. In particular, 
\begin{enumerate}
\item for $v_p=[e_1,e_2,e_1+e_2,p,e_3]$ with the rank matrix $\varphi[4^2;5]$ for some $p\in (\mathbb P^1)'$, we have 
\begin{align*}\overline{GL_4\cdot v_p}=&GL_2\cdot v_p\ \amalg\ \coprod_{q\in\mathbb P^1}GL_4\cdot[e_1,e_2,e_1+e_2,p,q] \\
&\ \amalg\ \coprod_{i=1}^4\varpi^{-1}\left(\{i\}^1,\{5\}^1,[4]\setminus\{i\}^1\right) \ \amalg\ \coprod_{i=1}^4 \pi^{-1}\left(\varphi[5^2;[4]\setminus\{i\}]\right)\\ &\ \amalg\ \coprod_{i=1}^4\varpi^{-1}\left(\{i,j\}^1,[5]\setminus\{i,j\}^1\right) \amalg\ \coprod_{i=1}^5\varpi^{-1}\left(\{i\}^1,[5]\setminus\{i\}^1\right)\ \amalg\ \pi^{-1}\left(\{[5]^1\}\right),
\end{align*}
\item for a rank matrix $\varphi$ of type $(4,4)$b, the closure of the orbit $\pi^{-1}\left(\varphi\right)$ is $\coprod_{\psi\leq\varphi}\pi^{-1}(\psi)$.  
\end{enumerate}
\end{lemma}

\subsubsection{Orbits of the splitting type $\{5^3\}$}
For the unique splitting $\{[5]^3\}\in \mathrm{Image}(\varpi)\setminus\mathcal P_{4,5}'$ of type $\{5^3\}$, there are $4$ types of rank matrices $(5,10)$, $(5,8)$, $(5,6)$ and $(4,6)$, according to Theorem \ref{thm:parameter}. 

\bigskip
\noindent
\underline{{\bf Rank type $(5,10)$: }}

\bigskip
Consider the rank matrix $\varphi[5^3]\colon I\mapsto \min\{\#I,3\}$ of type $(5,10)$. Then an orbit of this type is represented as $GL_4\cdot [e_1,e_2,e_3,e_1+e_2+e_3,p]$ where $p=[p_1e_1+p_2e_2+p_3e_3]\in (\mathbb P^2)'$. 

From now on, for the polynomial ring on $(3\times 5)$-matrices $X$, the notion $|i,j,k|$ denotes the polynomial defined by the $3$-minor of $X$ consisting of the $i,j$, and $k$-th columns, which is homogeneous for each columns. Then we define a homogeneous ideal $I_p$ generated by the following $5$ polynomials. 
\begin{align}
P_1(X)&=p_3|1,2,4||1,5,3|+p_2|1,2,5||1,3,4|,\notag\\ P_2(X)&=p_3|2,1,4||2,5,3|+p_1|2,1,5||2,3,4|,\notag\\ P_3(X)&=p_2|3,1,4||3,5,2|+p_1|3,1,5||3,2,4| ,\label{eq:53}\\
P_4(X)&=(p_2-p_3)|4,1,3||4,5,2|+(p_1-p_3)|4,1,5||4,2,3|,\notag\\ P_5(X)&=p_1(p_2-p_3)|5,1,3||5,4,2|+p_2(p_1-p_3)|5,1,4||5,2,3|\notag
\end{align}
Since these polynomials are relatively invariant under the $GL_3$-action, the ideal $I_p$ defines a $GL_3$-stable closed subset $Z(I_p)$ of $(\mathbb P^2)^5$. Clearly we have $v_p\cdot [e_1,e_2,e_3,e_1+e_2+e_3,p]\in Z(I_p)$, hence $\overline{GL_3\cdot v_p}\subset Z(I_p)$. 
We shall show the converse inclusion. For this purpose, at first we determine the set $Z(I_p)$ according to the classification of orbits. 

\begin{enumerate}
\item First, we consider the case where $v$ is of rank $3$ and all triples of columns are linearly independent, or equivalently, the rank matrix of $v$ is $\varphi[5^3]$. In this case, $v$ is in the orbit $GL_3\cdot v_q$ for some $q=[q_1e_1+q_2e_2+q_3e_3]\in (\mathbb P^2)'$. 
Then we have $P_1(v)=q_2p_3-q_3p_2$, $P_2(v)=q_1p_3-q_3p_1$, $P_3(v)=q_1p_2-q_2p_1$. They all coincides with $0$ if and only if $p=q$. 
Hence in this case, $GL_3\cdot v_p$ with the rank matrix $\varphi[5^3]$ of type $(5,10)$ is the only orbit contained in $Z(I_p)$. 

\item Next, we focus on the case where $v=[v_i]_{i=1}^5$ is of rank $3$, but there exists a vanishing $3$-minor $|i,j,k|$ which is in a general position in a $2$-dimensional space. We can set $v_k=v_i+v_j$ by renormalising the vectors. Now, remark that the $5$ polynomials are of the form  
\begin{align*}
P_i(X)&=c_1|i,j,l||i,m,k|+c_2|i,j,m||i,k,l|\\
&=-c_1|i,j,k||i,l,m|+(c_2-c_1)|i,j,m||i,k,l| \\
&=(c_1-c_2)|i,j,l||i,m,k|-c_2|i,j,k||i,l,m|\end{align*}
for $\{i,j,k,l,m\}=[5]$ and $c_1c_2(c_1-c_2)\neq 0$. Hence if $v\in Z(I_p)$, then $|i,j,k|=0$ leads that either $v_l$ or $v_m$ is in the $2$-dimensional space $\langle v_i,v_j,v_k\rangle=\langle v_i,v_j\rangle$. We may assume it to be $v_l$ without loss of generality, and we can set $v_l=q_iv_i+q_jv_j$. Furthermore, since $v$ is of rank $3$, we have $v_m\notin \langle v_i,v_j,v_k,v_l\rangle=\langle v_i,v_j\rangle$. Under these observations, we have 
\[P_m(v)=c_1|m,i,k||m,l,j|+c_2|m,i,l||m,j,k|=|m,i,j|^2\left(c_1q_i-c_2q_j\right),  \]
which leads that $[v_l]=[c_2v_i+c_1v_j]$. Hence in this case, all orbits contained in $Z(I_p)$ are through either the followings: 
\begin{gather} \label{eq:42}
\begin{split}
\left[\begin{array}{ccccc}0&1&0&1&p_2\\0&0&1&1&p_3\\1&0&0&0&0\end{array}\right],\ \ \ \left[\begin{array}{ccccc}1&0&0&1&p_1\\0&0&1&1&p_3\\0&1&0&0&0\end{array}\right],\ \ \ \left[\begin{array}{ccccc}1&0&0&1&p_1\\0&1&0&1&p_2\\0&0&1&0&0\end{array}\right], \\
\left[\begin{array}{ccccc}1&0&1&0&p_1-p_3\\0&1&1&0&p_2-p_3\\0&0&0&1&0\end{array}\right],\ \ \ \left[\begin{array}{ccccc}1&0&1&p_2(p_1-p_3)&0\\0&1&1&p_1(p_2-p_3)&0\\0&0&0&0&1\end{array}\right]. 
\end{split}
\end{gather} 
Conversely, these orbits are contained in $Z(I_p)$ clearly, because $P_m$ vanishes from the computation above, and the other polynomials automatically vanish since their terms are divided by some $3$-minors contained in $i,j,k,l$-th columns. 
Hence in this case, the orbits in $Z(I_p)$ are fulfilled with ones through the elements in (\ref{eq:42}), whose splitting type is $\{1^1,4^2\}$ and rank matrices are $\varphi[4^2;m]$ for $m\in[5]$ of type $(5,5)$.

\item If $v$ is of rank $3$, and there exists a pair $\{v_i,v_j\}$ which is proportional to each others. Then we can take $[v_i]=[v_j]$, $[v_k]$, $[v_l]$ to be linearly independent, and we have  
\begin{align*}
P_k(v)&=c_1|k,i,j||k,l,m|+c_2|k,i,l||k,m,j|=c|i,k,l||i,k,m|, \\
P_l(v)&=d_1|l,i,j||l,k,m|+d_2|l,i,k||l,m,j|=d|i,k,l||i,l,m|, 
\end{align*}
for some $cd\neq 0$. Hence $v_m\in\langle v_i,v_k\rangle\cap \langle v_i,v_l\rangle=\langle v_i\rangle$, so the vectors $v_i,v_j,v_m$ are proportional to each others. Conversely, if $3$ columns are proportional to each others, then it is contained in $Z(I_p)$ since all terms of each polynomials can be divided by $3$-minors including two of these columns. 
Hence in this case, all orbits in $Z(I_p)$ are fulfiled with the splitting type $\{1^1,1^1,3^1\}$, which corresponds to the rank type $(3,3)$b. 

\item The rest case is where $v$ is of rank less than $3$. Then it is included in $Z(I_p)$ since all $3$-minors vanish. 
\end{enumerate} 
Now, we shall show that $Z(I_p)$ is included in the closure of $GL_3\cdot v_p$. 
\begin{enumerate}
\item Consider the $G$-invariant map 
\[\begin{array}{ccc}
\left(\mathbb P^2\times \mathbb P^2\times\mathbb P^2\times \mathbb P^2\right)' & \to & (\mathbb P^2)^5 \\
\rotatebox{90}{$\in$} && \rotatebox{90}{$\in$} \\
v=[v_i]_{i=1}^4 &\mapsto& \left[v_1,v_2,v_3,v_4,p_1|2,3,4|v_1+p_2|3,1,4|v_2+p_3|1,2,4|v_3\right]
\end{array}\]
where $\left(\mathbb P^2\times \mathbb P^2\times\mathbb P^2\times \mathbb P^2\right)'$ denotes the open dense subset of $(\mathbb P^2)^4$ consisting of all elements of rank $3$. Then the preimage of $GL_3\cdot v_p$ consists of elements such that $4$ vectors are in general positions, which is an open dense subset of $\left(\mathbb P^2\times \mathbb P^2\times\mathbb P^2\times \mathbb P^2\right)'$. Hence the whole image of this map is included in the closure of $GL_2\cdot v_p$. 

Consider the element $v=[e_3,e_1,e_2,e_1+e_2]\in \left(\mathbb P^2\times \mathbb P^2\times\mathbb P^2\times \mathbb P^2\right)'$ of splitting type $\{1^1,3^2\}$, then the value of this map is the first element in (\ref{eq:42}) of type $\{1^1,4^2\}$. By permuting indices of the map and the element $v$, we also obtain other elements in (\ref{eq:42}) as the values of this map similarly. 

On the other hand, consider the element $v=[v_i]_{i=1}^4\in \left(\mathbb P^2\times \mathbb P^2\times\mathbb P^2\times \mathbb P^2\right)'$ such that $[v_i]=[v_j]$ for some $1\leq i\neq j\leq 4$, whose splitting type is $\{1^1,1^1,2^1\}$. Then the value of this map is of the splitting $\{\{i,j,5\}^1,\{k\}^1,\{l\}^1\}$ where $\{i,j,k,l\}=[4]$. By permuting the indices of the map, we also obtain all orbits of splitting type $\{3^1,1^1,1^1\}$. 

\item Define an open dense subset of $(\mathbb P^1)^2$ consisting of all elements $(q,r)\in(\mathbb P^1)^2$ satisfying that the following numbers are not $0$.  
\begin{gather*}
c_1:=\left|r\ \ \left(\begin{array}{cc}p_3&\\p_3-p_2&p_2\end{array}\right)q\right|, \ \ \ c_2:=\left|r\ \ \left(\begin{array}{cc}p_1&p_3-p_1\\&p_3\end{array}\right)q\right|,\\ c_3:=-\left|r\ \ \left(\begin{array}{cc}p_1&\\&p_2\end{array}\right)q\right|, 
\ \ c_4:=\left|r\ \ \begin{array}{c}p_3-p_1\\p_3-p_2\end{array}\right|,\ \ \ c_5:=\left|q\ \ \begin{array}{c}p_2(p_3-p_1)\\p_1(p_3-p_2)\end{array}\right|. 
\end{gather*}
If we set a linear endomorphism on $\mathbb K^3$ by $e_1\mapsto c_1e_1$, $e_2\mapsto c_2e_2$, and $e_3\mapsto c_3(e_1+e_2)$, then we have
\begin{align*}
e_1+e_2+e_3\mapsto & (c_1+c_3)e_1+(c_2+c_3)e_2
=c_4(q_1e_1+q_2e_2) \\
p_1e_1+p_2e_2+p_3e_3\mapsto &(p_1c_1+p_3c_3)e_1+(p_2c_2+p_3c_3)e_2=c_5(r_1e_1+r_2e_2), 
\end{align*}
which leads that $GL_4\cdot [e_1,e_2,e_1+e_2,q,r]$ is included in the closure of $GL_3\cdot [e_1,e_2,e_3,e_1+e_2+e_3,p]$ where $(q,r)$ is in the open dense subset of $(\mathbb P^1)^2$ defined above. Hence same argument holds for any $(q,r)\in (\mathbb P^1)^2$. Furthermore, since $\coprod_{(q,r)\in (\mathbb P^1)^2}GL_4\cdot [e_1,e_2,e_1+e_2,q,r]=\left\{v=[v_i]_{i=1}^5\left|\mathrm{rank}\ v=2,\ [v_1]\neq [v_2]\neq[v_3]\neq[v_1]\right.\right\}$ is an open dense subset of $\left\{v=[v_i]_{i=1}^5\left|\mathrm{rank}\ v\leq 2\right.\right\}$, we have shown that any elements of rank less than $3$ is contained in  the closure of $GL_3\cdot [e_1,e_2,e_3,e_1+e_2+e_3,p]$. 
\end{enumerate} 

\begin{lemma}\label{thm:cl53}
 For the orbit $GL_4\cdot v_p:=GL_4\cdot [e_1,e_2,e_3,e_1+e_2+e_3,p]$ with the rank matrix $\varphi[5^3]$ of type $(5,10)$ where $p\in(\mathbb P^2)'$, the closure of the orbit is the zero locus of the ideal generated by the polynomials in (\ref{eq:53}), and 
\[\overline{GL_4\cdot v_p}=GL_4\cdot v_p\ \amalg\ \coprod_{i=1}^5 GL_4\cdot v_p^{(i)}\ \amalg\ \coprod_{i\neq j}\varpi^{-1}\left(\{\{i\}^1,\{j\}^1,[5]\setminus\{i,j\}^1\right)\ \amalg\ \left\{v\in(\mathbb P^3)^5\left|\ \mathrm{rank}\ v\leq 2\right.\right\}\]
where $v_p^{(i)}\in(\mathbb P^3)^5$ is the element listed in (\ref{eq:42}) of the splitting $\{\{i\}^1,[5]\setminus\{i\}^2\}$ and the rank matrix $\varphi[4^2;i]$ of type $(5,5)$.
\end{lemma}

\noindent
\underline{{\bf Rank type $(5,8)$: }} 

\bigskip
Next we consider the orbits with the rank matrices $\varphi[5^3;J]$ of type $(5,8)$ where $\#J=3$. For simplicity, we set $J=[3]$. Then an orbit of this type is represented as $GL_4\cdot [e_1,e_2,p,e_3,e_1+e_2+e_3]$ with $p=[p_1e_1+p_2e_2]\in (\mathbb P^1)'$ from Theorem \ref{thm:parameter}. Now, consider the ideal $I_p$ generated by the following $3$ polynomials: 
\begin{align}
P(X)&=|1,2,3| \notag\\
P_4(X)&=p_1|4,1,3||4,5,2|+p_2|4,1,5||4,2,3| \label{eq:533}\\
P_5(X)&=p_1|5,1,3||5,4,2|+p_2|5,1,4||5,2,3|. \notag
\end{align}
It is clear that the closure of the orbit through $[e_1,e_2,p,e_3,e_1+e_2+e_3]$ is included in the $GL_4$-stable closed subset $Z(I_p)$. We shall show the converse inclusion. 

The condition that the first polynomial vanishes for $[v_i]_{i=1}^5\in(\mathbb P^2)^5$ is equivalent to that $\{v_i\}_{i=1}^3$ is linearly dependent, hence we assume it from now on.  

If $v$ is of rank $3$ and $\dim \langle v_1,v_2,v_3\rangle=\dim\langle v_4,v_5\rangle=2$. Then either $v_4$ or $v_5$ is not in $\langle v_1,v_2,v_3\rangle$, hence we set $v_i\notin \langle v_1,v_2,v_3\rangle$ where $\{i,j\}=\{4,5\}$. Consider the decomposition $v_j=w_j+cv_i$ according to the direct sum $\langle v_1,v_2,v_3\rangle\oplus \langle v_i\rangle$. Then we have 
\begin{align*}
P_i(v)&=p_1|i,1,3||i,j,2|+p_2|i,1,j||i,2,3|=d\left(p_1|v_1\ v_3||w_j\ v_2|+p_2|v_1\ w_j||v_2\ v_3|\right) \\
P_j(v)&=p_1|j,1,3||j,i,2|+p_2|j,1,i||j,2,3|=-cd\left(p_1|v_1\ v_3||w_j\ v_2|+p_2|v_1\ w_j||v_2\ v_3|\right) 
\end{align*}
where the determinants in the right-hand side are defined under some basis $\{x,y\}$ of $\langle v_1,v_2,v_3\rangle$ and $d=|v_i\ x\ y|\neq 0$. Hence in this case, $v$ is in $Z(I_p)$ if and only if $[v_1,v_2,w_j,v_3]\in \overline{GL_3\cdot[e_1,e_2,e_1+e_2,p]}$ (see Lemma \ref{thm:cl522}). 
\begin{enumerate}
\item If $[v_1,v_2,w_j,v_3]\in GL_3\cdot[e_1,e_2,e_1+e_2,p]$, then $[v_1,v_2,v_3,v_i,v_j]\in GL_3\cdot [e_1,e_2,p,e_3,e_1+e_2+ce_3]$. If $c\neq 0$, then we have $[v_1,v_2,v_3,v_4,v_5]\in GL_3\cdot [e_1,e_2,p,e_3,e_1+e_2+e_3]$, which is the orbit with the rank matrix $\varphi[5^2;[3]]$ of type $(5,8)$ which we are dealing with. On the other hand, if $c=0$, then $[v_1,v_2,v_3,v_i,v_j]\in GL_3\cdot [e_1,e_2,p,e_3,e_1+e_2]$ with the splitting $\{\{i\}^1,[5]\setminus \{i\}^2\}$ and the rank matrix $\varphi[4^2;i]$ of type $(5,5)$ where $i=4,5$. 
\item Since we assumed that $\dim\langle v_1,v_2,v_3\rangle=2$, the rest case is where $[v_k]\neq [v_l]=[v_m]=[w_j]$ for $[3]=\{k,l,m\}$. Then we have $[v_k,v_l,v_m,v_i,v_j]\in GL_3\cdot [e_1,e_2,e_2,e_3,e_2+ce_3]$. The case $c\neq 0$ corresponds to the splitting $\{\{k\}^1,[5]\setminus\{k\}^2\}$ and the rank matrix $\varphi[4^2;k;\{l,m\}]$ of type $(4,4)$b where $\{k,l,m\}=[3]$. On the other hand, if $c=0$, then $[v_k,v_l,v_m,v_i,v_j]\in GL_3\cdot [e_1,e_2,e_2,e_3,e_2]$ and it is the orbit with the splitting $\{\{k\}^1,\{i\}^1,\{l,m,j\}^1\}$ and the rank type $(3,3)$b where $\{i,j\}=\{4,5\}$. 
\end{enumerate}
Remark that if $v$ is either of rank $2$, $\dim\langle v_4,v_5\rangle =1$, or $\dim\langle v_1,v_2,v_3\rangle =1$ then the $3$ polynomials vanish clearly. 
In particular, the case $[v_4]=[v_5]$ corresponds to the splittings $\{\{4,5\}^1,[3]^2\}$ of rank type $(4,4)$a, and $\{\{4,5\}^1,\{k,l\}^1,\{m\}^1\}$ of type $(3,3)$a where $\{k,l,m\}=[3]$. The case $[v_1]=[v_2]=[v_3]$ corresponds to the splitting $\{[3]^1,\{4\}^1,\{5\}^1\}$ of type $(3,3)$b. 
According to these arguments, we shall show that $Z(I_p)$ is included in the closure of the orbit $GL_4\cdot [e_1,e_2,p,e_3,e_1+e_2+e_3]$. 
\begin{enumerate}
\item Consider the map 
\[\begin{array}{ccc}
\left(\mathbb P^2\times\mathbb P^2\times\mathbb P^2\times\mathbb P^2\right)' &\to & \left(\mathbb P^2\right)^5 \\
\rotatebox{90}{$\in$} &&\rotatebox{90}{$\in$} \\
\left[v_1,v_2,v_4,v_5\right] &\mapsto &[v_1,v_2,p_1|2,4,5|v_1+p_2|4,1,5|v_2,v_4,v_5]
\end{array}\]
where $\left(\mathbb P^2\times\mathbb P^2\times\mathbb P^2\times\mathbb P^2\right)' $ denotes the open dense subset of $\left(\mathbb P^2\right)^4$ consisting of all elements of rank $3$ and $[v_4]\neq [v_5]$. The preimage of the orbit $GL_4\cdot [e_1,e_2,p,e_3,e_1+e_2+e_3]$ is the open dense subset of $\left(\mathbb P^2\times\mathbb P^2\times\mathbb P^2\times\mathbb P^2\right)' $ consisting of all elements such that $\{v_1,v_2,v_4,v_5\}$ are in general positions in $3$-dimensional spaces. Hence the whole image of this map is included in the closure of the orbit. 

The correspondence between $4$-tuples of vectors and the value of the map is given as follows by a simple computation. 
\begin{table}[h]
\centering
\begin{tabular}{c|c}
$\{v_1,v_2,v_j\}$: general where $\{i,j\}=\{4,5\}$ & rank matrix $\varphi[4^2;i]$ \\\hline 
$\{v_l,v_4,v_5\}$: general where $\{k,l\}=\{1,2\}$ & rank matrix $\varphi[4^2;k;\{l,3\}]$ \\\hline
$[v_1]=[v_2]$ & splitting $\{\{4\}^1,\{5\}^1,[3]^1\}$ \\\hline
$[v_l]=[v_j]$ where $\{k,l\}=\{1,2\}$, $\{i,j\}=\{4,5\}$ & splitting $\{\{k\}^1,\{i\}^1,\{l,3,j\}^1\}$
\end{tabular}
\end{table}
By permuting indices $\{1,2,3\}$ of the map, we can also obtain the orbits with the rank matrices $\varphi[4^2;k;\{l,m\}]$ and splittings $\{\{k\}^1,\{i\}^1,\{l,m,j\}^1\}$ for $\{k,l,m\}=[3]$ and $\{i,j\}=\{4,5\}$. 

\item The rest cases where $v$ is of rank $3$ are the splittings $\{[3]^2,\{4,5\}^1\}$ and $\{\{k\}^1,\{l,m\}^1,\{4,5\}^1\}$ where $\{k,l,m\}=[3]$. 
Remark that $[e_1,e_2,p,e_3,c(e_1+e_2)+e_3]$ is in the orbit $GL_4\cdot[e_1,e_2,p,e_3,e_1+e_2+e_3]$ where $c\neq 0$. Hence the orbit through $[e_1,e_2,p,e_3,e_3]$, whose splitting is $\{[3]^2,\{4,5\}^1\}$, is included in the closure. 
Furthermore, since the orbit with the splitting type $\{1^1,2^1\}$ is included in the closure of the orbit with the splitting type $\{3^2\}\in\mathcal P_{2,3}'$ from Proposition \ref{thm:clsplit}, the orbit with the splitting $\{\{k\}^1,\{l,m\}^1,\{4,5\}^1\}$ is also contained in the closure. 

\item Let us consider an open dense subset of $(\mathbb P^1)^2$ consisting of all elements $(q,r)$ satisfying $\tilde p\neq q\neq r\neq \tilde p$ where $\tilde p=[p_2e_1+p_1e_2]$. Consider the linear map sending $e_1\mapsto p_2|r,q|e_1$, $e_2\mapsto p_1|r,q|e_2$, $e_3\mapsto|\tilde p,r|q$, then it sends the vectors $p\mapsto p_1p_2|r,q|(e_1+e_2)$ and $e_1+e_2+e_3\mapsto |\tilde p,q|r$, which leads that the orbit through $[e_1,e_2,e_1+e_2,q,r]$ is included in the closure of $GL_4\cdot[e_1,e_2,
p,e_3,e_1+e_2+e_3]$ if $(q,r)$ is in this open dense subset of $(\mathbb P^1)^2$. Hence same argument holds for any $(q,r)\in (\mathbb P^1)^2$, and all elements $v\in(\mathbb P^2)^5$ with rank less than $3$ are clearly contained in the closure. 
\end{enumerate}

\begin{lemma} \label{thm:cl533}
For an orbit $GL_4\cdot v_p=GL_4\cdot[e_1,e_2,p,e_3,e_1+e_2+e_3]$ with the rank matrix $\varphi[5^3;[3]]$ of type $(5,8)$ where $p\in (\mathbb P^1)'$, the closure of the orbit is the zero locus the ideal generated by the polynomials in (\ref{eq:533}), and
\begin{align*}
\overline{GL_4\cdot v_p}&=GL_4\cdot v_p\ \amalg\ \coprod_{i=4,5}GL_4\cdot v^{(i)}_p\ \amalg\ \coprod_{\{k,l,m\}=[3]}\pi^{-1}\left(\varphi[4^2;k;\{l,m\}]\right)\ \amalg\ \varpi^{-1}\left(\{[3]^2,\{4,5\}^1\}\right)\\
&\ \ \ \amalg\ \coprod_{\{i,j\}\nsubseteq [3]}\varpi^{-1}\left(\{\{i\}^1,\{j\}^1,[5]\setminus\{i,j\}^1\}\right)\ \amalg\ \coprod_{\{k,l,m\}=[3]}\varpi^{-1}\left(\{\{k\}^1,\{l,m\}^1,\{4,5\}^1\}\right) \\
&\ \ \ \amalg\ \left\{v\in (\mathbb P^3)^5\left|\ \mathrm{rank}\ v\leq 2\right.\right\}.  
\end{align*}
where $v_p^{(4)}=[e_1,e_2,p,e_3,e_1+e_2]$ and $v_p^{(5)}=[e_1,e_2,p,e_1+e_2,e_3]$ with the spltting $\{\{i\}^1,[5]\setminus\{i\}^2\}$ and the rank matrix $\varphi[4^2;i]$ of type $(5,5)$. 
\end{lemma}

\noindent
\underline{{\bf Rank type $(5,6)$: }}

\bigskip
Consider the rank matrix $\varphi[5^3;J_1,J_2]$ of type $(5,6)$ where $\#J_1=\#J_2=2$ and $J_1\cap J_2=\emptyset$. Recall that the fibre $\pi^{-1}\left(\varphi[5^3;J_1,J_2]\right)$ is a single orbit, and 
\[\varphi[5^3;J_1,J_2]\colon I\mapsto \begin{cases}\min\{\#I,3\} & I\neq [5]\setminus J_1,\ [5]\setminus J_2 \\ 2 & I= [5]\setminus J_1,\ [5]\setminus J_2.\end{cases}\]
Hence the orbit $\pi^{-1}\left(\varphi[5^3;J_1,J_2]\right)$ is a open subset of the closed set
\begin{equation} \label{eq:5333}
\coprod_{\varphi\leq\varphi[5^3;J_1,J_2]}\pi^{-1}(\varphi)=\left\{[v_i]_{i=1}^5\in(\mathbb P^3)^5\left|\ \dim\langle v_i\rangle_{i\notin J_1},\dim\langle v_i\rangle_{i\notin J_2}\leq 2,\textrm{ and }\dim\langle v_i\rangle_{i=1}^5\leq 3\right.\right\}. \end{equation}
Consider the following map 
\[M(4,3)'\times (\mathbb P^1)^4\to(\mathbb P^3)^5,\ \left((x\ y\ z),[p_i]_{i=2}^5\right)\mapsto \left[x,(x\ y)p_2,(x\ y)p_3,(x\ z)p_4,(x\ z)p_5\right] \]
where $M(4,3)'$ denotes the open dense subset of $M(4,3)$ consisting of all $4\times 3$-matrices $(x\ y\ z)$ such that all columns are not $\bm 0$ and $[x]\neq [y],[z]$. It is clear that the image of this map is included in the set (\ref{eq:5333}) where $J_1=\{2,3\}$ and $J_2=\{4,5\}$. In fact, the converse also holds. Indeed, let $v=[v_i]_{i=1}^5$ be in the set (\ref{eq:5333}), set $x:=v_1$.   
\begin{enumerate}
\item If $[v_1]=[v_2]=[v_3]$, then take some $[y]\neq [x]$ and $[p_2],[p_3]:=[e_1]$. 
\item If $[v_1]\neq [v_2]$, then set $y:=v_2$ and $p_2:=[e_2]$. Furthermore, set $p_3$ according to the linear combination $v_3=p_{31}x+p_{32}y$. The case where $[v_1]\neq[v_3]$ can be similarly considered. 
\end{enumerate}
For $\{v_3,v_4\}$, we define similarly, then we obtain $v$ as a value of the map. Hence the closed set (\ref{eq:5333}) is irreducible, and we obtain the following: 
\begin{lemma}\label{thm:5333} For the orbit $\pi^{-1}(\varphi[5^3;J_1,J_2])$ of rank type $(5,6)$ where $\#J_1=\#J_2=2$ and $J_1\cap J_2\neq \emptyset$, we have
$\overline{\pi^{-1}(\varphi[5^3;J_1,J_2])}=\coprod_{\varphi\leq \varphi[5^3;J_1,J_2]}\pi^{-1}(\varphi).$
\end{lemma}

\noindent
\underline{{\bf Rank type $(4,6)$: }}

\bigskip
The last rank matrices of the splitting $\{[5]^3\}$ are $\varphi[5^3;J]$ of the type $(4,6)$ where $\#J=2$. Recall that $\pi^{-1}(\varphi[5^3;J])$ are single orbits. 

Remark that $v=[v_i]_{i=1}^5$ is in this orbit if and only if the two columns corresponding to $J$ are proportional to each others, and the rest $4$ vectors are in a general position. Hence, it is equivalent to consider the orbit of the splitting $\{[4]^3\}\in\mathcal P_{4,4}'$ in $(\mathbb P^3)^4$. From Proposition \ref{thm:clsplit}, we have the following. 
\begin{lemma}\label{thm:cl532} For orbits $\pi^{-1}(\varphi[5^3;J])$ of rank type $(4,6)$ where $\#J=2$, we have 
\[\overline{\pi^{-1}(\varphi[5^3;J])}=\coprod_{\varphi\leq \varphi[5^3;J]}\pi^{-1}(\varphi).\] 
\end{lemma}

\subsection{Combinatorial criterion for closure relations}
In subsections \ref{clsplit}, \ref{clparameter}, we observed the following: 
\begin{thm} \label{thm:clsingle}
If $\pi^{-1}(\varphi)\subset(\mathbb P^3)^5$ is a single $GL_4$-orbit for $\varphi\in\mathrm{Image}(\pi)\subset Map(2^{[5]},\mathbb N)$, then $\overline{\pi^{-1}(\varphi)}=\coprod_{\psi\leq \varphi}\pi^{-1}(\psi)$. 
\end{thm}
However, if $\pi^{-1}(\varphi)$ is decomposed into infinitely many orbits, then the closure of each orbit shows a different phenomenon. For instance, the closure of an orbit of the rank matrix $\varphi[5^3]$ does not include orbits of the splitting type $\{1^1,2^1,2^1\}$ even the corresponding rank matrices are smaller than $\varphi[5^3]$. 

For these orbits, we determined the closures by a case-by-case computation in Subsection \ref{clparameter}. In this subsection, we give a general criterion to determine the closures, using the combinatorial aspects of rank matrices. 

\begin{df} \label{df:reduction} For a map $\varphi\in Map(2^{[m]},\mathbb N)$, we define the following notion. 
For a face $J\subset[m]$ of $\varphi$, in other words, if $\varphi(J)<\varphi(J')$ holds for all $J\subsetneq J'$, then we define the reduction $\varphi_J\in Map(2^{[m]},\mathbb N)$ of $\varphi$ by $J$ as $\varphi_J\colon I\mapsto \varphi(I\cup J)+\varphi(I\cap J)-\varphi(J)$. 
Furthermore, if we obtain a map $\psi$ by a reduction of $\varphi$ by $J$, then we write $\varphi\overset{J}{\rightsquigarrow}\psi$. 
\end{df}
\begin{exmp} 
For the rank matrix $ Map(2^{[5]},\mathbb N)\ni\varphi[5^3]\colon\ I\mapsto \min\{\#I,3\}$ of type $(5,10)$ with the splitting $\{[5]^3\}$, vertices are fulfilled with one point sets $\{i\}\subset[5]$. Then we have
\begin{align*}
\varphi[5^3]_{\{i\}}(I)&=\varphi[5^3](I\cap\{i\})+\varphi[5^3](I\cup\{i\})-\varphi[5^3](\{i\}) \\
&=\min\{\#(I\cap \{i\}),3\}+\min\{\#(I\cup\{i\}),3\}-1 \\
&=\min\{\#(I\cap \{i\}),1\}+\min\{\#(I\setminus\{i\}),2\}=\varphi[4^2;i](I). 
\end{align*}
Hence the reduction of $\varphi[5^3]$ by a vertex $\{i\}$ is $\varphi[4^2;i]$ of type $(5,5)$ with the splitting $\{\{i\}^1,[5]\setminus\{i\}^2\}$. 
On the other hand, edges of $\varphi[5^3]$ are fulfilled with two point sets $\{i,j\}\subset[5]$, and 
\begin{align*}
\varphi[5^3]_{\{i,j\}}(I)&=\varphi[5^3](I\cap\{i,j\})+\varphi[5^3](I\cup\{i,j\})-\varphi[5^3](\{i,j\})  \\
&=\min\{\#(I\cap \{i,j\}),3\}+\min\{\#(I\cup\{i,j\}),3\}-2 \\
&=\min\{\#(I\cap \{i\}),1\}+\min\{\#(I\cap \{j\}),1\}+\min\{\#(I\setminus\{i,j\}),1\}. 
\end{align*}
Hence the reduction of $\varphi[5^3]$ by an edge $\{i,j\}$ is the unique rank matrix of type $(3,3)$b with the splitting $\{\{i\}^1,\{j\}^1,[5]\setminus\{i,j\}^1\}$. 
\end{exmp}

\begin{remark}
Remark that if $\varphi([m])=r$ for $\varphi\in \mathrm{Image}(\pi)$, then there exists a unique $r$-face $[m]$ and no $r+1$-faces. Furthermore, 
$\varphi_{[m]}(I)=\varphi(I\cap[m])+\varphi(I\cup [m])-\varphi([m])=\varphi(I).  $
\end{remark}

\begin{lemma} \label{thm:preceq} For $\varphi\in \mathrm{Image}(\pi)$, consider an orbit $\mathcal O\subset\pi^{-1}(\varphi)$, then $\pi^{-1}(\psi)\cap \overline{\mathcal O}\neq\emptyset$ holds if $\varphi\rightsquigarrow\psi$. In particular, if $\pi^{-1}(\psi)$ is a single orbit, then it is included in $\overline{\mathcal O}$. 
\end{lemma}
\begin{proof}
For $v=[v_i]_{i=1}^m\in\pi^{-1}(\varphi)$ and an $r$-face $J$ of $\varphi$, consider a direct sum $\mathbb K^n=\langle v_i\rangle_{i\in J}\oplus W$, and the element $w=[w_i]_{i=1}^m\in (\mathbb P^{n-1})^m$ defined by 
\[w_i:=\begin{cases} v_i & i\in J \\ \textrm{projection of $v_i$ onto $W$} & i\notin J.  \end{cases}\]
Then we have 
\begin{align*}
\pi(w)(I)&= \dim\langle w_i\rangle_{i\in I}=\dim\langle v_i\rangle_{i\in I\cap J}+\dim\langle w_i\rangle_{i\in I\setminus J}\\
&=\varphi(I\cap J)+\dim\left(\langle v_i\rangle _{i\in J}\oplus \langle w_i\rangle_{i\in I\setminus J}\right)-\dim\langle v_i\rangle_{i\in J} \\
&=\varphi(I\cap J)+\dim\langle v_i\rangle_{i\in I\cup J}-\varphi(J) \\
&=\varphi(I\cap J)+\varphi(I\cup J)-\varphi(J)=\varphi_J(I). 
\end{align*}
Hence $\pi(w)=\varphi_J$. Furthermore, define an element $v^c=[v_i^c]_{i=1}^m\in(\mathbb P^{n-1})^m$ for $c\in\mathbb K$ by 
\[v_i^c:=\begin{cases} v_i & i\in I \\ c(v_i-w_i)+w_i & i\notin J. \end{cases}\]
Then for $c\neq 0$, considering a linear isomorphism $g_c\in GL_n$ by $v\mapsto cv$ in $\langle v_i\rangle_{i\in J}$ and identity in $W$, we have $g_cv=v^c\in \mathcal O$. Hence we have $v^0=w\in \overline{\mathcal O}$, and obtain $\pi^{-1}(\varphi_J)\cap\overline{\mathcal O}\neq\emptyset$. 
\end{proof}

\begin{thm}\label{thm:clparameter} For rank matrices $\varphi$ of types $(5,10)$, $(5,8)$, $(5,5)$, $(5)$, or $(4)$ and orbits $\mathcal O\subset\pi^{-1}(\varphi)$, we have the followings for $\psi\in \mathrm{Image}(\pi)\subset Map(2^{[5]},\mathbb N)$: 
\begin{enumerate}
\item $\pi^{-1}(\psi)\cap \overline{\mathcal O}\neq\emptyset$ if and only if $\psi\preceq\varphi$ (see Definition \ref{df:partial}). 
\item $\pi^{-1}(\psi)\subset \overline{\mathcal O}$ if and only if $\psi\preceq\varphi$ and $\pi^{-1}(\psi)$ is a single orbit, or $\psi\prec\varphi$. 
\end{enumerate}
\end{thm}

\begin{proof}
First, we show that $\psi\preceq\varphi$ implies $\pi^{-1}(\psi)\cap \overline{\mathcal O}\neq\emptyset$. 
By definition of $\preceq$ in Definition \ref{df:partial}, it suffies to show the claim in the following cases. 
\begin{enumerate}
\item If $\psi$ is obtained by a reduction of $\varphi$, the claim holds from Lemma \ref{thm:preceq}. 
\item If $\rho(\varphi)=\{I_k^{r_k}\}_{k=1}^l$, $\psi\leq\varphi$, and $\psi_{I_k}=\varphi|_{I_k}$ for all $k\in [l]$, then the possibility of $\psi$ for each $\varphi$ of type $(5,10)$, $(5,8)$, $(5)$, or $(4)$ is $\varphi$ itself, because $\varphi$ is indecomposable. On the other hand, the possibility of $\psi$ for $\varphi=\varphi[4^2;i]\colon I\mapsto \#(\{i\}\cap I)+\min\{\#(I\setminus\{i\}),2\}$ of rank type $(5,5)$ is either $\varphi[5^2]$ or $\varphi[5^2;\{i,j\}]$ for $i\neq j$, whose fibre intersects with $\overline{\mathcal O}$. 
\end{enumerate}
Remark that in the cases above, if we additionally assume that $\pi^{-1}(\psi)$ is a single orbit, then $\pi^{-1}(\psi)\cap\overline{ \mathcal O}\neq\emptyset$ implies $\pi^{-1}(\psi)\subset\overline{\mathcal O}$. 
\begin{enumerate}
\setcounter{enumi}{2}
\item Consider the case where $\psi\prec \varphi$, in other words, assume that $\rho(\varphi)=\{I_k^{r_k}\}_{k=1}^l$, $\rho(\psi)=\{I_k^{s_k}\}_{k=1}^l$, $\psi\leq\varphi$, and there exists some $k\in[l]$ such that $s_k<r_k$. 

Now, since the splitting of $\varphi$ with the rank matrix of type $(5,10)$ or $(5,8)$ is $\{[5]^3\}$, then the splitting of $\psi$ has to be $\{[5]^2\}$ or $\{[5]^1\}$, of which all orbits are included  in $\overline{\mathcal O}$. Same argument holds for other cases. 
\end{enumerate}
Combining these results, we also have shown that, if $\psi\preceq\varphi$ and $\pi^{-1}(\psi)$ is a single orbit, or if $\psi\prec\varphi$, then $\pi^{-1}(\psi)\subset\overline{\mathcal O}$. 
Next, we show the converse direction. 
\begin{enumerate}
\item Consider $\varphi=\varphi[5^3]$ of type $(5,10)$. From Lemma \ref{thm:cl53}, if $\pi^{-1}(\psi)\subset \overline{\mathcal O}$, then the splitting type of $\psi$ is either $\{5^2\}$, $\{5^1\}$, $\{1^1,4^1\}$, $\{2^1,3^1\}$ or $\{1^1,1^1,3^1\}$, and if $\pi^{-1}(\psi)\cap  \overline{\mathcal O}\neq \emptyset$, then $\psi=\varphi[4^2;i]$ is also allowed. 

If $\psi$ is of the splitting type $\{5^2\}$ or $\{5^1\}$, then $\psi\prec\varphi[5^3]$. Furthermore, we have
\[\varphi[5^2]\overset{\{i\}}{\rightsquigarrow}\rho|^{-1}\left(\{\{i\}^1,[5]\setminus\{i\}^1\}\right),\ \ \ \varphi[5^2;\{i,j\}]\overset{\{i,j\}}{\rightsquigarrow}\rho|^{-1}\left(\{\{i,j\}^1,[5]\setminus\{i,j\}^1\}\right), \]
hence $\psi\preceq \varphi[5^3]$ holds for $\psi$ of the splitting type $\{1^1,4^1\}$, or $\{2^1,3^1\}$, and $\pi^{-1}(\psi)$ is a single orbit in this case. 

The rest splitting type of $\psi$ is only $\{1^1,4^2\}$ or $\{1^1,1^1,3^1\}$. In this case, we have 
\[
\varphi[5^3]\overset{\{i\}}{\rightsquigarrow} \varphi[4^2;i]\overset{\{j\}}{\rightsquigarrow} \rho|^{-1}\left(\{\{i\}^1,\{j\}^1,[5]\setminus\{i,j\}^1\}\right) \]
for $i\neq j\in[5]$. Hence $\psi\preceq \varphi[5^3]$, and in the latter case, $\pi^{-1}(\psi)$ is a single orbit. 

\item Consider $\varphi=\varphi[5^3;J]$ of type $(5,8)$ where $\#J=3$. From Lemma \ref{thm:cl533}, if $\pi^{-1}(\psi)\subset \overline{\mathcal O}$, then the splitting of $\psi$ is either of type $\{5^2\}$, $\{5^1\}$, $\{1^1,4^1\}$, $\{2^1,3^1\}$, or $\{J^2,[5]\setminus J^1\}$, $\{\{i\}^1,J\setminus\{i\}^1,[5]\setminus J^1\}$ with $i\in J$, $\{\{i\}^1,\{j\}^1,[5]\setminus\{i,j\}^1\}$ with $\{i,j\}\nsubseteq J$, or $\psi=\varphi[4^2;i;J\setminus\{i\}]$ where $i\in J$. If $\pi^{-1}(\psi)\cap  \overline{\mathcal O}\neq \emptyset$, then $\psi=\psi[4^2;l]$ with $l\notin J$ is also allowed. 

If $\psi$ is of the splitting type $\{5^2\}$, $\{5^1\}$, $\{1^1,4^1\}$, or $\{2^1,3^1\}$, then the same argument holds as in the previous case. 

For the rest cases, we have 
\begin{align*}
\varphi[5^3;J]&\overset{J}{\rightsquigarrow} \rho|^{-1}\left(\{J^2,[5]\setminus J^1\}\right)\overset{\{i\}\subset J}{\rightsquigarrow} \rho|^{-1}\left(\{\{i\}^1,J\setminus\{i\}^1,[5]\setminus J^1\}\right) \\
\varphi[5^3;J]&\overset{\{i\}\subset J}{\rightsquigarrow} \varphi[4^2;i;J\setminus\{i\}]\overset{\{l\}\nsubseteq J}{\rightsquigarrow} \rho|^{-1}\left(\{\{i\}^1,\{l\}^1,[5]\setminus\{i,l\}^1\}\right) \\
\varphi[5^3;J]&\overset{\{l\}\notin J}{\rightsquigarrow} \varphi[4^2;l]\overset{\{m\}\nsubseteq J\amalg\{l\}}{\rightsquigarrow} \rho|^{-1}\left(\{\{l\}^1,\{m\}^1,J^1\}\right). 
\end{align*}
Hence $\psi\preceq \varphi[5^2;J]$ holds, and $\pi^{-1}(\psi)$ is a single orbit unless $\psi=\varphi[4^2;l]$ for $l\notin J$. 

\item Consider $\varphi=\varphi[4^2;i]$ of type $(5,5)$ where $i\in [5]$. From Lemma \ref{thm:cl42}, if $\pi^{-1}(\psi)\subset \overline{\mathcal O}$, then the splitting of $\psi$ is either of splitting type $\{5^1\}$, $\{1^1,4^1\}$, or $\{J^1,[5]\setminus J^1\}$ with $\#J=2$ and $i\in J$, $\{\{i\}^1,\{j\}^1,[5]\setminus\{i,j\}^1\}$ with $j\neq i$, or $\psi=\varphi[5^2;J]$ with $\#J=3$ and $i\notin J$. If $\pi^{-1}(\psi)\cap  \overline{\mathcal O}\neq \emptyset$, then $\psi=\psi[5^2]$, $\psi[5^2;J]$ with $\#J=2$ and $i\in J$ are also allowed. 

Since $\varphi[4^2;i]|_{[5]\setminus\{i\}}\colon I\mapsto \min\{\#I,2\}$, we have $\varphi[5^2],\varphi[5^2;J]\preceq \varphi[4^2;i]$ for $i\in J$. 

For the rest cases, we have $\varphi[4^2;i]\overset{\{j\}\neq \{i\}}{\rightsquigarrow}\rho|^{-1}\left(\{\{i\}^1,\{j\}^1,[5]\setminus\{i,j\}^1\}\right)$. Since \[\rho|^{-1}\left(\{\{i\}^1,\{j\}^1,[5]\setminus\{i,j\}^1\}\right)\colon I\mapsto \#(I\cap \{i\})+\#(I\cap \{j\})+\#(I\setminus\{i,j\}),\] clearly we have $\psi\preceq\varphi$ in these cases, and $\pi^{-1}(\psi)$ is a single orbit. 

\item Consider $\varphi=\varphi[5^2]$ of type $(5)$. From Lemma \ref{thm:cl52}, $\pi^{-1}(\psi)\subset \overline{\mathcal O}$ is equivalent to $\pi^{-1}(\psi)\cap  \overline{\mathcal O}\neq \emptyset$, and then the splitting of $\psi$ is either of splitting type $\{5^1\}$ or $\{1^1,4^1\}$. The former case satisfies $\psi\prec \varphi$, and the latter case is obtained by a reduction by a vertex, and $\pi^{-1}(\psi)$ is a single orbit. Similar argument holds for $\varphi=\varphi[5^2;J]$ with $\#J=2$. 
\end{enumerate}
\end{proof}
This completes the proof of Theorem \ref{thm:closure}.

\end{document}